\documentclass{amsart}
\usepackage{amssymb, amsmath, amsthm, graphics, comment, xspace, enumerate}
\theoremstyle{plain}

\theoremstyle{definition}
\usepackage[]{graphicx}
\newtheorem{theorem}{Theorem}[section]

\newtheorem{lem}[theorem]{Lemma}
\newtheorem{proposition}[theorem]{Proposition}
\theoremstyle{definition}
\newtheorem{definition}[theorem]{Definition}

\theoremstyle{remark}
\newtheorem{remark}[theorem]{Remark}
\numberwithin{equation}{section}
\begin{document}
\keywords{$C_{0}$-Semigroups on complex sectors, ${\mathcal F}$-frequent hypercyclicity, $f$-frequent hypercyclicity, translation semigroups and semigroups induced by semiflows, Fr\' echet spaces}
\subjclass[msc2010]{47A16, 47B37, 47D06}%



\title[$f$-Frequently hypercyclic $C_{0}$-semigroups on complex sectors]{$f$-Frequently hypercyclic $C_{0}$-semigroups on complex sectors}

\author{Belkacem Chaouchi}
\address{Lab. de l'Energie et des Syst\' emes Intelligents, Khemis Miliana University,
 44225 Khemis Miliana, Algeria}
\email{chaouchicukm@gmail.com}

\author{Marko Kosti\' c}
\address{Faculty of Technical Sciences,
University of Novi Sad,
Trg D. Obradovi\' ca 6, 21125 Novi Sad, Serbia}
\email{marco.s@verat.net}

\author{Stevan Pilipovi\' c}
\address{Department for Mathematics and Informatics,
University of Novi Sad,
Trg D. Obradovi\' ca 4, 21000 Novi Sad, Serbia}
\email{pilipovic@dmi.uns.ac.rs}

\author{Daniel Velinov}
\address{Department for Mathematics, Faculty of Civil Engineering, Ss. Cyril and Methodius University, Skopje,
Partizanski Odredi
24, P.O. box 560, 1000 Skopje, Macedonia}
\email{velinovd@gf.ukim.edu.mk}

\begin{abstract}
  We analyze $f$-frequently hypercyclic, $q$-frequently hypercyclic ($q> 1$) and frequently hypercyclic $C_{0}$-semigroups ($q=1$) defined on complex sectors, working in the setting of separable infinite-dimensional Fr\'echet spaces. Some structural results of ours are given for a general class of ${\mathcal F}$-frequently hypercyclic $C_{0}$-semigroups, as well. We investigate generalized frequently hypercyclic translation semigroups
and generalized frequently hypercyclic semigroups induced by semiflows on weighted function spaces. Several illustrative examples are presented.
\end{abstract}
\maketitle                   






\section{Introduction and preliminaries}

The main aim of this paper is to investigate the classes of ${\mathcal F}$-frequently hypercyclic $C_{0}$-semigroups and $f$-frequently hypercyclic
$C_{0}$-semigroups on complex sector $\Delta$, in the setting of separable infinite-dimensional Fr\'echet spaces.  Here, we assume that $\Delta \in  \{ [0,\infty),\
{\mathbb R},\ {\mathbb C} \}$ or $\Delta=\Delta(\alpha)$ for an
appropriate angle $\alpha \in (0,\frac{\pi}{2}],$ where $\Delta(\alpha):=\{ re^{i\theta} : r\geq 0,\ \theta \in
[-\alpha ,\alpha] \}.$ Further on,  ${\mathcal F}\in P(P(\Delta)),$ where $P(\Delta)$ denotes the power set of $\Delta ,$ ${\mathcal F}\neq \emptyset $ and $f : [0,\infty) \rightarrow [1,\infty)$ is an increasing mapping.

In the case that $\Delta=[0,\infty),$ the classes of ${\mathcal F}$-frequently hypercyclic integrated $C$-semigroups and $f$-frequently hypercyclic
integrated $C$-semigroups have been recently introduced and analyzed in \cite{C-DS}; although there is a real possibility to analyze the corresponding classes of integrated $C$-semigroups
with the index set $\Delta \neq [0,\infty),$ we will focus our attention henceforth to $C_{0}$-semigroups, only. In the case that $q\geq 1$ and $f(t):=t^{q}+1,$ $t\geq 0,$ an $f$-frequently hypercyclic $C_{0}$-semigroup $(T(t))_{t\in \Delta}$ is also said to be $q$-frequently hypercyclic, while the usual frequent hypercyclicity is obtained by plugging $q=1.$
The notion of a frequently hypercyclic $C_{0}$-semigroup $(T(t))_{t\in \Delta}$ on a separable Fr\'echet space $E$ seems to be not considered elsewhere even in the case that $\Delta=[0,\infty)$ and, because of that, we have been forced to reconsider and slightly extend several known results from \cite{man-marina}-\cite{man-peris} to $C_{0}$-semigroups in Fr\'echet spaces; the established results about frequently hypercyclic $C_{0}$-semigroups defined on complex sector $\Delta \neq [0,\infty)$ are completely new. A great deal of our work is devoted to the study of much more general class of $f$-frequently hypercyclic $C_{0}$-semigroups on complex sectors, which is the central object of our investigations, and for which we can freely say that it is completely unexplored by now.

The notion of a frequently hypercyclic linear continuous operator $T$ on a separable Fr\' echet space $E$
was introduced in \cite{bay1} and after that analyzed by many other authors; see \cite{bay1}-\cite{boni-ERAT} and \cite{Grosse}-\cite{erdos}. The most common technique for proving frequent hypercyclicity of $T$ is to show that $T$ satisfies the famous Frequent Hypercyclicity Criterion \cite{boni}-\cite{boni-ERAT}. The notion of a frequently hypercyclic $C_{0}$-semigroup on a separable Banach space was introduced in \cite{man-peris}, while the class of frequently hypercyclic translation $C_{0}$-semigroups on weighted function spaces has been examined in \cite{man-marina}.
On the other hand, hypercyclic and chaotic translation $C_{0}$-semigroups on complex sectors were investigated  in \cite{kuckanje}-\cite{kuckanje-prim}. The strong influential factor for genesis of this paper presents the fact that the notion of a frequently hypercyclic $C_{0}$-semigroup on a complex sector has not yet been defined in the existing literature, even in the setting of separable Banach spaces. For the sake of brevity and better exposition, we will not consider here the notion of strong mixing measures for frequently hypercyclic $C_{0}$-semigroups (cf. \cite{mere}) and some specific subnotions of frequent hypercyclicity connected with the use of weighted densities
that are sharper than the natural  lower density (see \cite{besinjo}, \cite{ernst}).

The organization and main ideas of this paper are briefly described as follows.
After giving some preliminaries about $C_{0}$-semigroups on Fr\'echet spaces, we analyze lower and upper densities in Subsection \ref{mareka}, and remind ourselves of the basic facts about  weighted function spaces, translation semigroups and semigroups induced by semiflows in Subsection \ref{srbo}. Conjugacy lemma for ${\mathcal F}$-frequently hypercyclic $C_{0}$-semigroups on complex sectors is stated in Lemma \ref{conjugacy}. In Proposition \ref{sot}-\ref{manic}, we generalize the implication (iii) $\Rightarrow$ (i) of \cite[Proposition 2.1]{man-peris}
for ${\mathcal F}$-frequently hypercyclic $C_{0}$-semigroups and $f$-frequently hypercyclic $C_{0}$-semigroups
on complex sectors. The main purpose of Proposition \ref{markknop} is to show that for any
hypercyclic $C_{0}$-semigroup $(T(t))_{t\in \Delta}$ we can always find an increasing mapping $f : [0,\infty) \rightarrow [1,\infty)$ satisfying that $(T(t))_{t\in \Delta}$ is $f$-frequently hypercyclic.
In
Theorem \ref{oma}, which is commonly used in applications, we formulate
sufficient conditions for $f$-frequent hypercyclicity of a single operator $T(t_{0}),$ where $t_{0}\in \Delta \setminus \{0\},$ as well the whole semigroup $(T(t))_{t\in \Delta}$ and its restriction to the ray $R$ connecting the origin and $t_{0}.$
The Frequent Hypercyclic Criterion for $C_0$-semigroups \cite[Theorem 2.2]{man-peris} cannot be satisfactorily reformulated for $q$-frequent hypercyclicity, provided that $q>1.$ Essentially, our first contribution is Theorem \ref{f-fhc}, where we formulate and prove $f$-Frequent Hypercyclic Criterion for $C_{0}$-semigroups on complex sectors by following an approach which do not use
the notion of Pettis integrability. Theorem \ref{oma-niz} is an extension of Theorem \ref{oma} for sequences of single operators of considered $C_{0}$-semigroup and this theorem provides an important tool for proving the existence of a frequently hypercyclic semigroup $(T(t))_{t\in \Delta(\pi/4)}$ without any (frequently) hypercyclic single operator; see also Examples I in Subsection \ref{zajeb}, in which we reconsider several known examples from \cite{kuckanje}-\cite{kuckanje-prim} and \cite[Chapter 3]{knjigaho}. In Theorem \ref{f-freq-21.}, we state an important extension of \cite[Proposition 2.1]{man-peris} for $f$-frequently hypercyclic semigroups with the index set $\Delta=[0,\infty).$

The third section  is devoted to the study of $f$-frequently hypercyclic translation semigroups
and $f$-frequently hypercyclic semigroups induced by semiflows on weighted function spaces. The method presented in the proofs of Theorem \ref{dov} and Theorem \ref{dovs}, in which we profile sufficient conditions for $f$-frequent hypercyclicity of translation semigroups on complex sectors, is repeated after that several times with minor modifications (obtaining necessary and sufficient conditions is a non-trivial problem, even for $q$-hypercyclicity). These results can be slightly improved by taking into analysis the sequences of single operators of considered $C_{0}$-semigroups (see Theorem \ref{dov-niz}-Theorem \ref{dovs-niz}). As explained in Remark \ref{hnme}, these results are applicable for $q$-frequent hypercyclicity, where $q>1,$ and not for the usual frequent hypercyclicity.
In Theorem \ref{nasser} and Proposition \ref{qwea}, we slightly extend the assertions of \cite[Theorem 6]{kuckanje} and \cite[Proposition 3.4]{man-marina}, respectively, concerning  frequently hypercyclic $C_{0}$-semigroups. In our examinations of $C_{0}$-semigroups induced by semiflows, we essentially use the same methods as for translation $C_{0}$-semigroups so that we shorten the proofs to a large extent or only explain how these proofs can be deduced.
It is worth noting that frequently hypercyclic weighted composition $C_{0}$-semigroups with the index set $\Delta=[0,\infty)$ have been considered in \cite{kalmes}, where the author founded certain conditions under which these semigroups satisfy the Frequent Hypercyclicity Criterion for $C_{0}$-semigroups  iff
they are chaotic (an application has been made to the linear von Foerster-Lasota equation in the spaces $L^{p}[0,1]$ and $W^{1,p}_{\ast}[0,1]$). Concerning the abstract first order differential equations
with a non-vanishing zero-order term, whose solutions are governed by weighted composition $C_{0}$-semigroups, we would like to note that the assertions of \cite[Theorem 3.1.25, Theorem 3.1.27]{knjigaho} hold for the general notion of ${\mathcal F}$-hypercyclicity and its subnotions; in other words, our main structural results established here for the abstract first order differential equations
without a non-vanishing zero-order term can be almost technically modified for the abstract first order differential equations
containing a non-vanishing zero-order term. Because of that, we will skip all related details concerning this question here.
Chaotic and frequently hypercyclic weighted composition $C_{0}$-semigroups on complex sectors deserve special attention and these classes of $C_{0}$-semigroups will be further analyzed in our follow-up research \cite{composition}. In Section \ref{srboljube}, we also present a great number of illustrative examples, applications and open problems (see Subsection 3.1 and Subsections 3.4-3.6).

\subsection{Notation and definitions}\label{mareka} We use the standard notation throughout the paper. For any $s\in {\mathbb R},$ set $\lfloor s \rfloor :=\sup \{
l\in {\mathbb Z} : s\geq l \}$ and $\lceil s \rceil :=\inf \{ l\in
{\mathbb Z} : s\leq l \}.$ By $E$ we denote a separable Fr\' echet space over the field ${\mathbb K}\in \{{\mathbb R},\ {\mathbb C}\}$; we assume that
the topology of $E$ is induced by the fundamental system
$(p_{n})_{n\in {\mathbb N}}$ of increasing seminorms. If $Y$ is also a Fr\' echet space, over the same field of scalars as $E,$ then by $L(E,Y)$ we denote the space consisting of all continuous linear mappings from $E$ into $Y.$ The
translation invariant metric $d: E\times E \rightarrow [0,\infty),$ defined by
$$
d(x,y):=\sum
\limits_{n=1}^{\infty}\frac{1}{2^{n}}\frac{p_{n}(x-y)}{1+p_{n}(x-y)},\quad
x,\ y\in E,
$$
satisfies, among many other properties, the following ones:
$d(x+u,y+v)\leq d(x,y)+d(u,v)$ and $d(cx,cy)\leq (|c|+1)d(x,y),\
c\in {\mathbb K},\ x,\ y,\ u,\ v\in X.$ Set $L(x,\epsilon):=\{y\in X : d(x,y)<\epsilon\},$ $\|x\|:=d(x,0)$ and $L_{n}(x,\epsilon):=\{y\in X : p_{n}(x-y)<\epsilon\}$ ($n\in {\mathbb N},$ $\epsilon>0$). Denote by $E^{\ast}$ the dual space of $E.$ For a closed linear operator $A$ on $E,$ we denote by $D(A),$ $R(A)$ and $N(A)$ its domain, range and kernel, respectively. 
 Set $D_{\infty}(A):=\bigcap_{k\in {\mathbb N}}D(A^{k}).$
Basic information about Bochner integration of Fr\'echet space valued functions can be obtained by consulting \cite[pp. 99-102]{martinez} and references cited therein.
Let us recall that a collection of series
$\sum_{n=1}^{\infty}x_{n,k},$ $k\in J$ is called
unconditionally convergent, uniformly in $k\in J$ iff for any $\epsilon>0$ there is an integer
$N\in {\mathbb N}$ such that for any finite set $F\subseteq [N,\infty) \cap {\mathbb N}$ and for every $k\in J$ we have $\sum_{n\in F}x_{n,k} \in L(0,\epsilon);$ see \cite{boni}, \cite{kadets}, \cite{robertson} and references cited therein for further information on the subject.
Set $\Delta_{\delta}:=\{ t\in \Delta : |t| \leq \delta\}$ and $B(t,\delta):=\{ x\in {\mathbb K} : |x-t| \leq \delta\}.$
We refer to \cite{kalmes} and \cite{knjigaho} for the basic material about $C_0-$ semigroups on Fr\' echet spaces. Within our framework,
any semigroup $(T(t))_{t\in \Delta}$ is locally equicontinuous, i.e., for any number $t_{0} > 0,$ the family
of operators $\{T(t) : t\in \Delta_{t_{0}}\}$ is equicontinuous. We will use the following definitions.



\begin{definition}\label{4-skins-MLO-okay} (\cite{1211212018})
Let $(T_{n})_{n\in {\mathbb N}}$ be a sequence of linear operators acting between the spaces
$E$ and $Y,$ let $T$ be a linear operator on $E$, and let $x\in E$. Suppose that ${\mathcal F}\in P(P({\mathbb N}))$ and ${\mathcal F}\neq \emptyset.$ Then it is said
that:
\begin{itemize}
\item[(i)] $x$ is an ${\mathcal F}$-hypercyclic element of the sequence
$(T_{n})_{n\in {\mathbb N}}$ iff
$x\in \bigcap_{n\in {\mathbb N}} D(T_{n})$ and for each open non-empty subset $V$ of $Y$ we have that
$$
S(x,V):=\bigl\{ n\in {\mathbb N} : T_{n}x \in V \bigr\}\in {\mathcal F} ;
$$
$(T_{n})_{n\in {\mathbb N}}$ is said to be ${\mathcal F}$-hypercyclic iff there exists an ${\mathcal F}$-hypercyclic element of
$(T_{n})_{n\in {\mathbb N}}$;
\item[(ii)] $T$ is ${\mathcal F}$-hypercyclic iff the sequence
$(T^{n})_{n\in {\mathbb N}}$ is ${\mathcal F}$-hypercyclic; $x\in D_{\infty}(T)$ is said to be
an ${\mathcal F}$-hypercyclic element of $T$ iff $x$ is an ${\mathcal F}$-hypercyclic element of the sequence
$(T^{n})_{n\in {\mathbb N}}.$
\end{itemize}
\end{definition}

\begin{definition}\label{prckojed} (\cite{1211212018})
Suppose that $q\in [1,\infty),$ $A \subseteq {\mathbb N}$ and $(m_{n})$ is an increasing sequence in $[1,\infty).$ Then we define:
\begin{itemize}
\item[(i)] The lower $q$-density of $A,$ denoted by $\underline{d}_{q}(A),$ through:
$$
\underline{d}_{q}(A):=\liminf_{n\rightarrow \infty}\frac{|A \cap [1,n^{q}]|}{n}.
$$
\item[(ii)] The lower $(m_{n})$-density of $A,$ denoted by $\underline{d}_{m_{n}}(A),$ through:
$$
\underline{d}_{{m_{n}}} (A):=\liminf_{n\rightarrow \infty}\frac{|A \cap [1,m_{n}]|}{n}.
$$
\end{itemize}
\end{definition}

Assume that $q\in [1,\infty)$ and $(m_{n})$ is an increasing sequence  in $[1,\infty).$ Consider the notion introduced in Definition \ref{4-skins-MLO-okay} with:
(i) ${\mathcal F}=\{A \subseteq {\mathbb N} : \underline{d}(A):= \underline{d}_{1}(A)>0\},$ (ii) ${\mathcal F}=\{A \subseteq {\mathbb N} : \underline{d}_{q}(A)>0\},$
(iii) ${\mathcal F}=\{A \subseteq {\mathbb N} : \underline{d}_{{m_{n}}}(A)>0\};$
then we say that $(T_{n})_{n\in {\mathbb N}}$ ($T$) is frequently hypercyclic, $q$-frequently hypercyclic and
l-$(m_{n})$-hypercyclic, respectively ($x$ is then called a frequently hypercyclic vector, $q$-frequently hypercyclic vector and
l-$(m_{n})$-hypercyclic vector of $(T_{n})_{n\in {\mathbb N}}$ ($T$), respectively).

Denote by $m(\cdot)$ the Lebesgue measure on $\Delta.$
The following definition, providing a continuous analogue of Definition \ref{prckojed}, will be sufficiently enough for our further work:

\begin{definition}\label{prckojed-prim} (\cite{1211212018})
Suppose that $q\in [1,\infty),$ $A\subseteq \Delta$ and $f : [0,\infty) \rightarrow [1,\infty)$ is an increasing mapping. Then we define:
\begin{itemize}
\item[(i)] The lower $qc$-density of $A,$ denoted by $\underline{d}_{qc}(A),$ through:
$$
\underline{d}_{qc}(A):=\liminf_{t\rightarrow \infty}\frac{m(A \cap \Delta_{t^{q}})}{t}.
$$
\item[(ii)] The lower $f$-density of $A,$ denoted by $\underline{d}_{f}(A),$ through:
$$
\underline{d}_{f} (A):=\liminf_{t\rightarrow \infty}\frac{m(A \cap \Delta_{f(t)})}{t}.
$$
\end{itemize}
\end{definition}

\subsection[Weighted function spaces, translation semigroups and...]{Weighted function spaces, translation semigroups and semigroups induced by semiflows}\label{srbo}

A measurable function
$\rho : \Delta \rightarrow (0,\infty)$ is said to be an
admissible weight function\index{admissible weight function} iff there exist constants $M\geq 1$ and
$\omega \in {\mathbb R}$ such that $\rho(t)\leq Me^{\omega
|t'|}\rho(t+t')$ for all $t,\ t'\in \Delta.$ For such a function
$\rho(\cdot),$ we introduce the following Banach spaces:
$$
L^{p}_{\rho}(\Delta, {\mathbb K}):=\bigl\{ u : \Delta \rightarrow
{\mathbb K} \ ; u(\cdot) \mbox{ is measurable and } ||u||_{p} <\infty \bigr\},
$$
where $p\in [1,\infty)$ and $||u||_{p}:=(\int _{\Delta}|u(t)|^{p}
\rho(t)\, dt)^{1/p},$ as well as
$$
C_{0,\rho}(\Delta, {\mathbb K}):=\Bigl\{u : \Delta \rightarrow {\mathbb
K} \ ; u(\cdot) \mbox{ is continuous and }\lim _{t \rightarrow
\infty}u(t)\rho(t)=0\Bigr\},
$$
with $||u||:=\sup _{t\in \Delta}|u(t)\rho(t)|.$

Assume that $\Delta' \in  \{ [0,\infty),\
{\mathbb R},\ {\mathbb C} \}$ or $\Delta'=\Delta(\alpha')$ for an
appropriate angle $\alpha' \in (0,\frac{\pi}{2}],$ and $\Delta'\subseteq \Delta.$
Then the translation semigroup
$(T(t))_{t\in \Delta'},$ given by
\begin{align*}
\bigl(T(t)f\bigr)(x):=f(x+t),\quad x\in \Delta,\ t\in \Delta',
\end{align*}
is strongly continuous on $
L^{p}_{\rho}(\Delta, {\mathbb K})$ and $C_{0,\rho}(\Delta, {\mathbb K}).$

Hypercyclic and chaotic $C_{0}$-semigroups induced by semiflows have been analyzed for the first time in \cite{kalmes}.
We will use the following definition:

\begin{definition}\label{semiflow} (\cite{kalmes}, \cite{knjigaho})
Let $n\in {\mathbb N}$ and $\Omega$ be an open non-empty subset
of ${{\mathbb R}^{n}}.$ A continuous mapping $\varphi : \Delta
\times \Omega \rightarrow \Omega$ is called a semiflow if\index{semiflow}
$\varphi(0,x)=x,\ x\in \Omega ,$
$$ 
\varphi(t+s,x)=\varphi(t,\varphi(s,x)),\quad t,\ s \in \Delta , \ x\in
\Omega \mbox{ and}
$$ 
$$ 
x\mapsto \varphi(t,x) \mbox{ is injective for all } t\in \Delta .
$$ 
\end{definition}

Denote by $\varphi(t,\cdot)^{-1}$ the inverse mapping of
$\varphi(t,\cdot),$ i.e.,
$$ 
y=\varphi(t,x)^{-1} \mbox{ iff } x=\varphi(t,y),\  t\in \Delta .
$$ 

We deal with
the Banach space $L^{p}_{\rho_{1}}(\Omega, {\mathbb K}),$ where $\rho_{1} :
\Omega \rightarrow (0,\infty)$ is a locally integrable function
and the norm of an element $f\in L^{p}_{\rho_{1}}(\Omega, {\mathbb K})$
is given by $||f||_{p}:=(\int _{\Omega}|f(x)|^{p}\rho_{1}(x)
\, dx)^{1/p}.$ With a little abuse of notation, by $C_{0,\rho_1}(\Omega, {\mathbb K})$ we denote the Banach space
consisting of all continuous functions $f : \Omega \rightarrow
{\mathbb K}$ satisfying that, for every $\epsilon>0,$ $\{x\in \Omega
: |f(x)|\rho_1(x)\geq \epsilon \}$ is a compact subset of $\Omega;$
at this place, $\rho_1 : \Omega \rightarrow (0,\infty)$ is an upper
semicontinuous function and the norm of an element $f\in
C_{0,\rho_1}(\Omega, {\mathbb K})$ is given by $||f||:=\sup
_{x\in \Omega}|f(x)|\rho_1(x).$
It is well-known that $C_{c}(\Omega, {\mathbb K}),$ the subspace of $C(\Omega, {\mathbb K})$
consisting of all compactly supported functions,
is dense in $L^{p}_{\rho}(\Omega, {\mathbb K})$ and
$C_{0,\rho}(\Omega
,{\mathbb K}).$ If the field ${\mathbb K}$ is clearly detemined, then we also write $C_{c}(\Omega)$ for $C_{c}(\Omega, {\mathbb K})$ and so on and so forth.

Given a number $t\in \Delta,$ a semiflow $\varphi : \Delta \times
\Omega \rightarrow \Omega$ and a function $f : \Omega \rightarrow
{\mathbb K},$ we define $T_{\varphi}(t)f : \Omega \rightarrow {\mathbb
K}$ by $(T_{\varphi}(t)f)(x):=f(\varphi(t,x)),\ x\in \Omega .$ Then
$T_{\varphi}(0)f=f,$
$T_{\varphi}(t)T_{\varphi}(s)f=T_{\varphi}(s)T_{\varphi}(t)f=T_{\varphi}(t+s)f,\
t,\ s \in \Delta$ and Brouwer's theorem implies $C_{c}(\Omega)
\subseteq T_{\varphi}(t)C_{c}(\Omega).$ Furthermore, we have the following lemma.

\begin{lem}\label{mars} (\cite{knjigaho})
\begin{itemize}
\item[(i)] Suppose that $\varphi : \Delta \times \Omega \rightarrow \Omega$ is a
semiflow and $\varphi(t,\cdot)$ is a locally Lipschitz continuous
function for all $t\in \Delta .$ Then \emph{(b)} implies \emph{(a)}, where:
\begin{itemize}
\item[(a)] $(T_{\varphi}(t))_{t\in \Delta}$ is a $C_{0}$-semigroup in $L^{p}_{\rho_{1}}(\Omega)$ and
\item[(b)] \begin{equation}\label{sufficient}
\exists M, \omega \in {\mathbb R}\  \forall t \in \Delta
: \rho_{1}(x) \leq  Me^{\omega
|t|}\rho_{1}(\varphi(t,x))|\det D\varphi (t,x)| \mbox{
a.e. }x\in \Omega.
\end{equation}
\end{itemize}
If, additionally, $\varphi(t,\cdot)^{-1}$ is locally Lipschitz
continuous for every $t\in \Delta,$ then the above are equivalent.
\item[(ii)] Let $\varphi : \Delta \times \Omega \rightarrow \Omega$ be a
semiflow. Then $(T_{\varphi}(t))_{t\in \Delta}$ is a $C_{0}$-semigroup in $C_{0,\rho}(\Omega)$ iff the following
holds:
\begin{itemize}
\item[(a)] $\exists M, \omega \in {\mathbb R} \, \,  \forall t \in \Delta ,
\, x\in \Omega : \rho(x) \leq  Me^{\omega |t|}\rho(\varphi(t,x))$ and
\item[(b)] for every compact set $K\subseteq \Omega$ and for every $\delta>0,$
$t\in \Delta:$
\begin{equation}\label{kompakt}
\varphi(t,\cdot)^{-1}(K) \ \cap \ \{ x\in \Omega : \rho(x)\geq
\delta \} \mbox{ is a compact subset of }\Omega.
\end{equation}
\end{itemize}
\end{itemize}
\end{lem}

\section[Generalized frequent hypercyclicity for $C_{0}$-semigroups on complex sectors]{Generalized frequent hypercyclicity for $C_{0}$-semigroups on complex sectors}\label{srboljub}

We start our work by proposing the following general definition:

\begin{definition}\label{4-skins}
Assume that $(T(t))_{t\in \Delta}$ is a $C_{0}$-semigroup on $E,$ and $x\in E$. Let ${\mathcal F}\in P(P(\Delta))$ and ${\mathcal F}\neq \emptyset.$ Then it is said
that $x$ is an ${\mathcal F}$-hypercyclic element (vector) of $(T(t))_{t\in \Delta}$ iff
for each open non-empty subset $V$ of $E$ we have
$$
S(x,V):=\bigl\{ t\in \Delta : T(t)x \in V \bigr\}\in {\mathcal F} .
$$
It is said that $(T(t))_{t\in \Delta}$ is ${\mathcal F}$-hypercyclic iff there exists an ${\mathcal F}$-hypercyclic element of $(T(t))_{t\in \Delta},$ while $(T(t))_{t\in \Delta}$ is said to be hypercyclic iff it is hypercyclic with ${\mathcal F}$ being the collection of all non-empty subsets of $P(\Delta),$ i.e., if for each open non-empty subset $V$ of $E$ we have
$
S(x,V)\neq \emptyset;$ in this case, $x$ is said to be a hypercyclic element of $(T(t))_{t\in \Delta}.$
\end{definition}

The following lemma has been proved in \cite[Theorem 3.1.2(i)]{knjigaho} for the case that $\Delta\in \{ [0,\infty), \Delta(\alpha)\},$ with a completely different proof given:

\begin{lem}\label{lemara}
Suppose that $(T(t))_{t\in \Delta}$ is a hypercyclic $C_{0}$-semigroup, $x\in E$ is a hypercyclic element of $(T(t))_{t\in \Delta}$ and $s>0.$ Then
the set $\{T(t)x: t\in \Delta \setminus \Delta_{s}\}$ is dense in $E.$
\end{lem}

\begin{proof}
Suppose that $x^{*}\in E^{\ast}$ and $\langle x^{\ast}, T(t)x \rangle=0$ for $|t|\geq s.$ Let $t_{0}\in \Delta$ and $|t_{0}|<s.$ Then there exists a sufficiently large real number $k>0$ such that $\|kT(t_{0})x\|>\max_{t\in \Delta_{s}}\|T(t)x\|.$ By this inequality and the fact that $x$ is a hypercyclic element of $(T(t))_{t\in \Delta},$ it readily follows that there exists a sequence $(t_{n})_{n\in {\mathbb N}}$ in $\Delta \setminus \Delta_{s} $ such that $\lim_{n\rightarrow \infty}T(t_{n})x=kT(t_{0})x$ so that $0=\lim_{n\rightarrow \infty}\langle x^{\ast}, T(t_{n})x\rangle  =\langle x^{\ast}, kT(t_{0})x\rangle$ and therefore $0=\langle x^{\ast}, T(t_{0})x\rangle .$ Since $t_{0}$ was arbitrary, we get that $0=\langle x^{\ast}, T(t)x\rangle $
for all $t\in \Delta.$ This implies that $x^{\ast}=0$ since $x^{\ast}$ annulates a dense set in $E,$ finishing the proof.
\end{proof}

In the sequel, we will use the following special cases of Definition \ref{4-skins}:

\begin{definition}\label{prckojedd}
Let $q\in [1,\infty),$ and let $f : [0,\infty) \rightarrow [1,\infty)$ be an increasing mapping. Assume that $(T(t))_{t\in \Delta}$ is a $C_{0}$-semigroup. Then it is said that:
\begin{itemize}
\item[(i)] $(T(t))_{t\in \Delta}$ is $q$-frequently hypercyclic iff there exists $x\in E$ ($q$-frequently hypercyclic vector of $(T(t))_{t\in \Delta}$) such that for each open non-empty subset $V$ of $E$ we have $\underline{d}_{qc}(\{ t\in \Delta : T(t)x \in V \})>0;$
\item[(ii)] $(T(t))_{t\in \Delta}$ is $f$-frequently hypercyclic iff there exists $x\in E$ ($f$-frequently hypercyclic vector of $(T(t))_{t\in \Delta}$) such that for each open non-empty subset $V$ of $E$ we have $\underline{d}_{f}(\{ t\in \Delta : T(t)x \in V \})>0.$
\end{itemize}
It is said that $(T(t))_{t\in \Delta}$ is frequently hypercyclic iff $(T(t))_{t\in \Delta}$ is $q$-frequently hypercyclic with $q=1.$
\end{definition}

Concerning the notion introduced in Definition \ref{4-skins} (and therefore, in Definition \ref{prckojedd}), we have the following conjugacy lemma; the proof is very simple and can be left to the interested readers:

\begin{lem}\label{conjugacy}
Assume that $(T(t))_{t\in \Delta}$ is a $C_{0}$-semigroup on $E,$ and $x\in E$. Let ${\mathcal F}\in P(P(\Delta))$ and ${\mathcal F}\neq \emptyset.$ Suppose that $F$ is a separable infinite-dimensional Fr\' echet space over the same field of scalars as $E,$ and
$(S(t))_{t\in \Delta}$ is a $C_{0}$-semigroup on $F.$ If $\Psi : E \rightarrow F$ is a continuous mapping with dense range, $x$ is an ${\mathcal F}$-hypercyclic vector of $(T(t))_{t\in \Delta}$ and $\Psi \circ T(t)=S(t) \circ \Psi$ for all $t\geq 0,$ then $\Psi x$ is an ${\mathcal F}$-hypercyclic vector of $(S(t))_{t\in \Delta}.$
\end{lem}

Suppose now that ${\mathcal F}\in P(P({\mathbb N}))$ and ${\mathcal F}\neq \emptyset.$ If ${\mathcal F}$ satisfies
\begin{itemize}
\item[(I):] $A\in{\mathcal F}$ and $A\subseteq B$ imply $B\in{\mathcal F},$
\end{itemize}
then it is said that
${\mathcal F}$ is a Furstenberg family; a proper Furstenberg family ${\mathcal F}$ is any Furstenberg family satisfying that $\emptyset \notin {\mathcal F}$ (see \cite{furstenberg} for more details).

Since the $C_{0}$-semigroup $(T(t))_{t\in \Delta}$ under our consideration is locally equicontinuous, we can argue as in the proofs of \cite[Proposition 2.1]{man-peris} and \cite[Proposition 2.6]{C-DS}, where $\Delta=[0,\infty),$ to deduce the following result:

\begin{proposition}\label{sot}
Let ${\mathcal F}$ be a Furstenberg family.
Suppose that $(T(t))_{t\in \Delta}$ is a $C_{0}$-semigroup on $E$ and $(t_{n})_{n\in {\mathbb N}}\subseteq \Delta \setminus \{0\}$ is a given sequence.
If $x\in  E$ is an ${\mathcal F}$-hypercyclic element of sequence $(T(t_{n}))_{n\in {\mathbb N}},$ then $x$ is an ${\mathcal F}'$-hypercyclic element of $(T(t))_{t\in \Delta},$ where
\begin{align*}
{\mathcal F}'=\Biggl\{ B\subseteq \Delta : (\exists A\in {\mathcal F})\, (\exists \delta_{0}>0)\, \bigcup_{k\in A}[B(t_{k},\delta_{0}) \cap \Delta] \subseteq B \Biggr\}.
\end{align*}
Furthermore, if there exists
a number $t_{0} \in \Delta \setminus \{0\}$
such that $t_{n}=nt_{0}$ for all $n\in {\mathbb N},$
then
$x$ is an ${\mathcal F}''$-hypercyclic element of $(T(te^{i\arg(t_{0})}))_{t\geq 0},$ where
\begin{align*}
{\mathcal F}''=\Biggl\{ B\subseteq [0,\infty) : (\exists A\in {\mathcal F})\, (\exists \delta_{0}>0)\, \bigcup_{k\in A}[k|t_{0}|,k|t_{0}|+\delta_{0}]  \subseteq B \Biggr\}.
\end{align*}
\end{proposition}

We will include the most relavant details of the proof of the following proposition for the sake of completeness.

\begin{proposition}\label{manic}
Suppose that $(T(t))_{t\in \Delta}$ is a $C_{0}$-semigroup on $E,$ $(t_{n})_{n\in {\mathbb N}}\subseteq \Delta \setminus \{0\}$ is a given sequence, $\delta >0,$
\begin{align}\label{krke-rke}
\bigl| t_{n}-t_{m}\bigr|\geq \delta\mbox{  for  }m\neq n,
\end{align}
$f : [0,\infty) \rightarrow [1,\infty)$ is an increasing mapping and $m_{k}=f(k),$ $k\in {\mathbb N}_{0}.$
Suppose that $g : [1,\infty) \rightarrow [1,\infty)$ is an increasing mapping and $|t_{n}|\leq g(n)$ for all $n\in {\mathbb N}.$
If $x\in  E$ is an
l-$(m_{n})$-hypercyclic element
of sequence $(T(t_{n}))_{n\in {\mathbb N}},$ then $x$ is a
$(g\circ f)$-frequently hypercyclic element of $(T(t))_{t\in \Delta}$. Furthermore, we have the following:
\begin{itemize}
\item[(i)] If $x$ is a frequently hypercyclic element of sequence $(T(t_{n}))_{n\in {\mathbb N}}$ ($f(t)=t+1,$ $t\geq 0$), then $x$ is a $g$-frequently hypercyclic element of $(T(t))_{t\in \Delta}.$
\item[(ii)] If $q\geq 1$ and $x$ is a $q$-frequently hypercyclic element of sequence $(T(t_{n}))_{n\in {\mathbb N}}$ ($f(t)=t^{q}+1,$ $t\geq 0$), then $x$ is a $(g\circ (\cdot^{q}+1))$-frequently hypercyclic element of $(T(t))_{t\in \Delta}.$
\item[(iii)] If there exists
a number $t_{0} \in \Delta \setminus \{0\}$
such that $t_{n}=nt_{0}$ for all $n\in {\mathbb N},$ and $x\in E$ is a frequently hypercyclic ($q$-frequently hypercyclic,
l-$(m_{n})$-hypercyclic) element of sequence $(T(t_{n}))_{n\in {\mathbb N}},$
then $x$ is a frequently hypercyclic ($q$-frequently hypercyclic,
$f$-hypercyclic) element of $(T(t))_{t\in \Delta}$ and $(T(te^{i\arg(t_{0})}))_{t\geq 0}$.
\end{itemize}
\end{proposition}

\begin{proof}
Let
$x\in  E$ be an l-$(m_{n})$-hypercyclic element
of sequence $(T(t_{n}))_{n\in {\mathbb N}}.$ In order to prove that $x$ is an
$f$-frequently hypercyclic element of $(T(t))_{t\in \Delta},$ let an open non-empty set $V$ of $E$ be given. The prescribed assumptions yield the existence of a positive constant $c>0$ and a strictly increasing sequence $(n_{k})_{k\in {\mathbb N}}$  in ${\mathbb N}$ such that
the interval $[1,m_{n_{k}}]$ contains at least $\lfloor cn_{k} \rfloor$ elements of set $\{n\in {\mathbb N} : T(t_{n})x\in V\},$ say $l_{1},\cdot \cdot \cdot ,l_{\lfloor cn_{k} \rfloor}.$
Hence, $T(t_{l_{1}})x\in V,\cdot \cdot \cdot ,T(t_{l_{\lfloor cn_{k} \rfloor}})x\in V,$ $|t_{l_{j}}|\leq g(m_{n_{k}})=(g\circ f)(n_{k})$ for $1\leq j\leq \lfloor cn_{k} \rfloor$ and
$|t_{l_{j}}-t_{l_{j'}}|\geq \delta,$
provided $l_{j}\neq l_{j'}$ and $1\leq j,\ j'\leq \lfloor cn_{k} \rfloor.$
Arguing similarly as in the proof of \cite[Proposition 2.1]{man-peris}, the strong continuity and local equicontinuity of $(T(t))_{t\in \Delta}$ together imply the existence of a positive real number $\delta \in (0,1)$ such that $T(t)x\in V$ for any $t\in \bigcup_{1\leq j\leq \lfloor cn_{k} \rfloor}B(t_{l_{j}},\delta).$
This simply implies the first
conclusion. The parts (i)-(ii) simply follow from this fact, while the proof of (iii) can be deduced by using (i)-(ii) and Proposition \ref{sot}.
\end{proof}

We continue by stating the following existence type result (a similar statement can be formulated for single operators, as well):

\begin{proposition}\label{markknop}
Assume that $(T(t))_{t\in \Delta}$ is a hypercyclic $C_{0}$-semigroup. Then there exists an increasing mapping $f : [0,\infty) \rightarrow [1,\infty)$ such that $(T(t))_{t\in \Delta}$ is $f$-frequently hypercyclic and, for every open non-empty subset $V$ of $E,$ we have
\begin{align}\label{tommyger}
\underline{d}_{f}(\{ t\in \Delta : T(t)x \in V \})=+\infty.
\end{align}
\end{proposition}

\begin{proof}
Let $(O_{n})_{n\in {\mathbb N}}$ be a base of the topology on $E,$ where $O_{n}\neq \emptyset$ for all $n\in {\mathbb N}.$ Assume that $x\in E$ is a hypercyclic vector of $(T(t))_{t\in \Delta}.$ Using Lemma \ref{lemara} as well as the hypercyclicity and strong continuity of $(T(t))_{t\in \Delta},$ we can simply prove that for each $n\in {\mathbb N}$ there exists a sequence $(t_{n,k})_{k\in {\mathbb N}}$ in $\Delta$ such that $T(t_{n,k})x\in O_{n}$ for all $k\in {\mathbb N}$ as well that $|t_{n,k}-t_{n',k'}|\geq 1$ for all $(n,k),\ (n',k')\in {\mathbb N}^{2}$ such that $(n,k)\neq (n',k').$ Define $F : {\mathbb N}^{2} \rightarrow {\mathbb N}$ by
$$
F(n,k):=\frac{(n+k-2)(n+k-1)}{2}+n,\quad (n,k)\in {\mathbb N}^{2}.
$$
Then $F(\cdot,\cdot)$ is a bijection together with its inverse mapping $F^{-1} : {\mathbb N} \rightarrow {\mathbb N}^{2}.$ We define $f_{0} : [0,\infty) \rightarrow [1,\infty)$
by $f_{0}(x):=f_{0}(\lceil x \rceil)$ for all $x\geq 0$, where the sequence $(f_{0}(N))_{N\in {{\mathbb N}_{0}}}$ is defined inductively as follows: Set $f_{0}(0):=1.$ For any $N\in {\mathbb N},$ we define $f_{0}(N)$ such that $f_{0}(N)>f_{0}(N-1)+1+\max_{1\leq j\leq N,1\leq k\leq N^{2}}F(n,k).$ It is clear that there exists an increasing mapping $g : [1,\infty) \rightarrow [1,\infty)$ such that $|t_{F^{-1}(N)}|\leq g(N)$ for all $N\in {\mathbb N}.$ Set $f:=g\circ f_{0}.$
We will prove that $(T(t))_{t\in \Delta}$ is $f$-frequently hypercyclic and \eqref{tommyger} holds
by applying Proposition \ref{manic} with the sequence $(T(t_{N}))_{N\in {\mathbb N}},$ where $t_{N}:=t_{F^{-1}(N)}$ for all $n\in {\mathbb N}.$
Evidently, the condition \eqref{krke-rke} holds with $\delta=1$ and $|t_{N}|\leq g(N)$ for all $N\in {\mathbb N}.$ Let $V$ be an open non-empty subset $V$ of $E.$ It is sufficient to prove that $x$ is an
l-$(m_{N})$-hypercyclic vector of the sequence $(T(t_{N}))_{N\in {\mathbb N}},$
as well as that for each $n\in {\mathbb N}$ there exists a strictly increasing sequence $(N_{n}^{k})_{k\in {\mathbb N}}$ of positive integers such that the interval $[1,f(N_{n}^{k})]$ contains at least $(N_{n}^{k})^{2}$ positive integers $m\in {\mathbb N}$ such that $F^{-1}(m)=(n,k)$ for some $k\in {\mathbb N}.$ But, this simply follows with the sequence $(N_{n}^{k}\equiv F(n,k))_{k\in {\mathbb N}},$ by using the fact that $f(N)>F(n,N^{2})$ for $1\leq n\leq N.$
\end{proof}

The following result is a simple consequence of \cite[Theorem 2.5]{C-DS}, in the case that $\Delta =[0,\infty).$ The proof in the case of general region $\Delta$ is similar and therefore omitted:

\begin{theorem}\label{oma}
Assume that $(T(t))_{t\in \Delta}$ is a $C_{0}$-semigroup on $E$ and $f : [0,\infty) \rightarrow [1,\infty)$ is an increasing mapping. Define $m_{k}:=f(k),$ $k\in {\mathbb N}_{0}.$
Suppose that there are a number $t_{0}\in \Delta \setminus \{0\},$ a dense subset $E_{0}$ of $E$ and mappings $S_{n} : E_{0} \rightarrow E$ ($n\in {\mathbb N}$) such that
the following conditions hold for all $y\in E_{0}$:
\begin{itemize}
\item[(i)] The sequence $\sum_{n=1}^{k}T(t_{0}\lfloor m_{k}\rfloor)S_{\lfloor m_{k-n}\rfloor}y, k\in\mathbb N$ converges unconditionally, uniformly in $k\in {\mathbb N}.$
\item[(ii)] The series $\sum_{n=1}^{\infty}T(t_{0}\lfloor m_{k}\rfloor)S_{\lfloor m_{k+n} \rfloor}y$ converges unconditionally, uniformly in $k\in {\mathbb N}.$
\item[(iii)] The series $\sum_{n=1}^{\infty}S_{\lfloor m_{n} \rfloor}y$ converges unconditionally.
\item[(iv)] $\lim_{n\rightarrow \infty}T(t_{0}\lfloor m_{n} \rfloor)S_{\lfloor m_{n} \rfloor}y=y.$
\end{itemize}
Then $(T(t))_{t\in \Delta}$ is $f$-frequently hypercyclic, $(T(te^{i\arg(t_{0})}))_{t\geq 0}$ is $f$-frequently hypercyclic and the operator $T(t_{0})$ is l-$(m_{k})$-frequently hypercyclic.
\end{theorem}

Setting $f(t):=t^{q}+1,$ $t\geq 0$ ($q\geq 1$), we obtain a sufficient condition for $q$-frequent hypercyclicity of $(T(t))_{t\in \Delta},$  $(T(te^{i\arg(t_{0})}))_{t\geq 0}$ and $T(t_{0}).$

In the following theorem, we state $f$-Frequent Hypercyclic Criterion for $C_{0}$-semigroups on complex sectors.

\begin{theorem}\label{f-fhc}
Assume that $(T(t))_{t\in \Delta}$ is a $C_{0}$-semigroup on $E,$ $f : [0,\infty) \rightarrow [1,\infty)$ is an increasing mapping and $m_{k}=f(k),$ $k\in {\mathbb N}_{0}.$ Set
$A_{k,n}^{f}:=[\lfloor m_{k+n} \rfloor -\lfloor m_{k} \rfloor -1, \lfloor m_{k+n} \rfloor -\lfloor m_{k} \rfloor ]$ ($k,\ n\in {\mathbb N}$)
and $B_{k,n}^{f}:=[\lfloor m_{k} \rfloor -\lfloor m_{k-n} \rfloor , \lfloor m_{k} \rfloor -\lfloor m_{k-n} \rfloor +1]$ ($k,\ n\in {\mathbb N},$ $k\geq n$).
Suppose that there exists a dense subset $E_{0}$ of $E$ such that $T(t)E_{0}\subseteq E_{0},$ $t\in \Delta$ and mappings $S(t) : E_{0}\rightarrow E_{0}$ ($t\in \Delta$) such that:
\begin{itemize}
\item[(i)] $T(t)S(t)x=x,$ $x\in E_{0},$ $t\in \Delta$ and $S(r)T(t)x=T(t)S(r)x=S(r-t)x,$ $x\in E_{0}$ if $r,\ t\in \Delta \setminus \{0\},$ $\arg(r)=\arg(t)$ and $|r|>|t|;$
\item[(ii)] For every $t_{0}\in \Delta \setminus \{0\}$ and $s\geq 0,$ the sequence $\sum_{n=1}^{k}\int_{s+B_{k,n}^{f}}T(tt_{0}e^{i\arg(t_{0})})x\, dt$, $n\in \mathbb N$ converges unconditionally, uniformly in $k\in {\mathbb N},$
for every $x\in E_{0}$.
\item[(iii)]  For every $t_{0}\in \Delta \setminus \{0\},$ the mapping $t\mapsto S(tt_{0}e^{i\arg(t_{0})})x,$ $t\in A_{k,n}^{f}$ is Bochner integrable for $k,\ n\in {\mathbb N},$ $x\in E_{0},$ and the series $\sum_{n=1}^{\infty}\int_{A_{k,n}^{f}}S(tt_{0}e^{i\arg(t_{0})})x\, dt$ converges unconditionally for $x\in E_{0}$, uniformly in $k\in {\mathbb N}.$
\end{itemize}
Then $(T(t))_{t\in \Delta}$ is $f$-frequently hypercyclic and, for every $t_{0}\in \Delta \setminus \{0\},$ $(T(te^{i\arg(t_{0})}))_{t\geq 0}$ is $f$-frequently hypercyclic and the operator $T(t_{0})$ is l-$(m_{k})$-frequently hypercyclic.
\end{theorem}

\begin{proof}
We will include the most relevant details of proof, which is  similar to that of
\cite[Proposition 2.8]{man-peris}. Let a number $t_{0}\in \Delta \setminus \{0\}$ be fixed. First of all, we will prove that the range of operator $I-T(t_{0})$ is dense in $E$. To verify this, assume that $x^{\ast} \in E^{\ast}$ and $\langle x^{\ast}, x-T(t_{0})x\rangle =0,$ $x\in E.$ Using the semigroup property, this inductively implies
\begin{align}\label{ovo}
\bigl \langle x^{\ast}, x-T(nt_{0})x \bigr \rangle =0,\quad x\in E,\ n\in {\mathbb N}.
\end{align}
Fix a number $s\geq 0.$
The continuity of $x^{\ast}$ and unconditional convergence of series\\ $\sum_{n=1}^{k}\int_{s+B_{k,n}^{f}}T(tt_{0}e^{i\arg(t_{0})})x\, dt,$ uniformly in $k\in {\mathbb N},$ imply that for each $\epsilon>0$ there exists $N\in {\mathbb N}$ such that, for every $k\in {\mathbb N}$ with $k\geq N,$ we have
$$
\Biggl| \int^{s+t_{0}(\lfloor m_{k} \rfloor -\lfloor m_{k-N} \rfloor +1)}_{s+t_{0}(\lfloor m_{k} \rfloor -\lfloor m_{k-N} \rfloor)}\bigl \langle x^{\ast} , T(r)x \bigr \rangle \, dr \Biggr| \leq \epsilon .
$$
Due to \eqref{ovo}, we have
$$
\int^{s+t_{0}(\lfloor m_{k} \rfloor -\lfloor m_{k-N} \rfloor +1)}_{s+t_{0}(\lfloor m_{k} \rfloor -\lfloor m_{k-n} \rfloor)}\bigl \langle x^{\ast} , T(r)x \bigr \rangle \, dr=\int^{s}_{0}\bigl \langle x^{\ast} , T(r)x \bigr \rangle \, dr=0,\quad x\in E,
$$
so that $\int^{s}_{0} \langle x^{\ast} , T(r)x  \rangle \, dr=0,$ $x\in E,$ $s\geq 0,$ which yields $x^{\ast}=0.$
Since the integral generator of semigroup $(T(te^{i\arg(t_{0})}))_{t\geq 0}$ is dense in $E$ (see e.g. \cite{komura}), we can proceed as in the proof of \cite[Theorem 2.2]{man-peris} to deduce that the set $E_{0}':=\{ \int^{t_{0}}_{0}T(t)x\, dt : x\in E_{0}\}$ is dense in $E.$ For any
$x\in E_{0}$ we have $y=\int^{t_{0}}_{0}T(t)x\, dt \in E_{0}'$ and, similarly as in the proof of \cite[Proposition 2.8]{man-peris}, the condition (i) and an elementary argumentation show that:
\begin{align*}
\sum_{n=1}^{k}T(t_{0}\lfloor m_{k}\rfloor)S_{\lfloor m_{k-n}\rfloor}y=t_{0}e^{i\arg(t_{0})}\sum_{n=1}^{k}\int_{\lfloor m_{k} \rfloor -\lfloor m_{k-n} \rfloor}^{\lfloor m_{k} \rfloor -\lfloor m_{k-n} \rfloor +1}T\bigl(tt_{0}e^{i\arg(t_{0})}\bigr)x\, dt,
\end{align*}
\begin{align*}
\sum_{n=1}^{\infty}T(t_{0}\lfloor m_{k}\rfloor)S_{\lfloor m_{k+n} \rfloor}y=t_{0}e^{i\arg(t_{0})}\sum_{n=1}^{\infty}\int_{\lfloor m_{k+n} \rfloor -\lfloor m_{k} \rfloor}^{\lfloor m_{k+n} \rfloor -\lfloor m_{k} \rfloor -1}S\bigl(tt_{0}e^{i\arg(t_{0})}\bigr)x\, dt
\end{align*}
and
\begin{align*}
\sum_{n=1}^{\infty}S_{\lfloor m_{n} \rfloor}y=t_{0}e^{i\arg(t_{0})}\sum_{n=1}^{\infty}\int_{\lfloor m_{n} \rfloor}^{\lfloor m_{n} \rfloor  -1}S\bigl(tt_{0}e^{i\arg(t_{0})}\bigr)x\, dt.
\end{align*}
The result now follows by applying Theorem \ref{oma}.
\end{proof}

Compared with \cite[Proposition 2.8]{man-peris},
the established criterion is not so nice-looking; besides that, there is no easy way to generalize \cite[Theorem 2.2]{man-peris} for $f$-frequent hypercyclicity.
Because of that,
we will focus our attention henceforth to Theorem \ref{oma} and its applications, primarily. But, considering of just one single operator $T(t_{0})$ may not be the ideal option for work in the case that $\Delta \neq [0,\infty),$ and we need to state the following slight modification of Theorem \ref{oma} (we can formulate a similar statement corresponding to Theorem \ref{f-fhc}, as well):

\begin{theorem}\label{oma-niz}
Assume that $(T(t))_{t\in \Delta}$ is a $C_{0}$-semigroup on $E,$ $\delta >0$ and $f : [0,\infty) \rightarrow [1,\infty)$ is an increasing mapping. Define $m_{k}:=f(k),$ $k\in {\mathbb N}_{0}.$
Suppose that there are a sequence $(t_{n})_{n\in {\mathbb N}}\subseteq \Delta \setminus \{0\},$ a dense subset $E_{0}$ of $E$ and mappings $S_{n} : E_{0} \rightarrow E$ ($n\in {\mathbb N}$) such that
the following conditions hold for all $y\in E_{0}$:
\begin{itemize}
\item[(i)]The series $\sum_{n=1}^{k}T(t_{\lfloor m_{k}\rfloor})S_{\lfloor m_{k-n}\rfloor}y, k\in\mathbb N$ converges unconditionally, uniformly in $k\in {\mathbb N}.$
\item[(ii)] The series $\sum_{n=1}^{\infty}T(t_{\lfloor m_{k}\rfloor})S_{\lfloor m_{k+n} \rfloor}y$ converges unconditionally, uniformly in $k\in {\mathbb N}.$
\item[(iii)] The series $\sum_{n=1}^{\infty}S_{\lfloor m_{n} \rfloor}y$ converges unconditionally.
\item[(iv)] $\lim_{n\rightarrow \infty}T(t_{\lfloor m_{n}\rfloor})S_{\lfloor m_{n} \rfloor}y=y,$ $y\in E_{0}.$
\item[(v)] The equation \eqref{krke-rke} holds and the increasing mapping $g : [1,\infty) \rightarrow [1,\infty)$ satisfies that $|t_{n}|\leq g(n)$ for all $n\in {\mathbb N}.$
\end{itemize}
Then the sequence of operators $(T_{n}:=T(t_{n}))_{n\in {\mathbb N}}$ is l-$(m_{k})$-frequently hypercyclic and
$(T(t))_{t\in \Delta}$ is $(g\circ f)$-frequently hypercyclic.
\end{theorem}

\begin{proof}
By \cite[Theorem 3.1]{1211212018}, we easily infer that the sequence $(T_{n})_{n\in {\mathbb N}}$ is l$-(m_{k})$-frequently hypercyclic.
Then the final statement follows from Proposition \ref{manic}.
\end{proof}

\subsection{Examples I}\label{zajeb}

We reconsider several illustrative examples proposed in \cite{kuckanje}-\cite{kuckanje-prim} and  \cite{knjigaho}:

\begin{itemize}
\item[(i)] (\cite[Example 2]{kuckanje}) We will prove that the translation semigroups $(T(t))_{t\in \Delta}$ and $(T(te^{-i\pi/4}))_{t\geq 0}$ are frequently hypercyclic as well as that any semigroup $(T(te^{i\varphi}))_{t\geq 0},$ where $-\pi/4<\varphi \leq \pi/4,$ is not hypercyclic.

Let $\Delta:=\Delta(\pi/4),$ let
\begin{align*}
\rho(x+iy):=
\left\{
\begin{array}{l}
1,\ x+y\geq \sqrt{x-y},\\
e^{x+y-\sqrt{x-y}},\mbox{ otherwise,}
\end{array}
\right.
\end{align*}
and let $E:=L^{p}_{\rho}(\Delta, {\mathbb K}).$
Then we know that the translation semigroups $(T(t))_{t\in \Delta}$ and $(T(te^{-i\pi/4}))_{t\geq 0}$ are chaotic as well as that any periodic point of $(T(t))_{t\in \Delta}$ lies on the boundary of $\Delta ;$ see \cite{kuckanje} for the notion.

 In order to see that $(T(te^{-i\pi/4}))_{t\geq 0}$ is frequently hypercyclic, we will apply the Frequent Hypercyclic Criterion for $C_{0}$-semigroups \cite[Theorem 2.2]{man-peris}. Let $E_{0}$ denote the set consisting of all continuous functions $f : \Delta \rightarrow {\mathbb K}$ with compact support; then $E_{0}$ is dense in $E.$
Define the mapping $S(t) f: \Delta \rightarrow {\mathbb K}$ by $(S(t) f)(x+iy):=f(x+iy-te^{-i\pi/4}),$ $x+iy\in te^{-i\pi/4}+supp(f)$ and $(S(t) f)(x):=0,$ otherwise ($f\in E_{0},$ $n\in {\mathbb N}$). Then $T(t)S(t)f=f,$ $f\in E_{0}$ and $T(t)S(r)f=S(r-t)f,$ $f\in E_{0},$ $r>t>0$, so that it suffices to show that the mappings $t\mapsto T(t)f,$ $t\geq 0$ and $t\mapsto S(t)f,$ $t\geq 0$ are Bochner integrable. This is clear for the first mapping because, for every $f\in E_{0},$ there exists $t_{0}>0$ such that $T(t)f=0$ for all $t\geq t_{0}.$ To see that the mapping $t\mapsto S(t)f,$ $t\geq 0$ is Bochner integrable for any $f\in E_{0}
$ with $supp(f)=K,$ observe first that:
\begin{align*}
\int^{\infty}_{0}\|S(t)f\|\, dt \leq \|f\|_{\infty}\int^{\infty}_{0}\Biggl( \int_{K}\rho\bigl( x+iy+te^{-i\pi/4}\bigr)\, dx\, dy \Biggr)^{1/p}\, dt .
\end{align*}
It is clear that there exist two finite numbers $t_{0}> 0$ and $c_{0}>0$ such that $x+y<\sqrt{x-y+t\sqrt{2}}$ for all $t\geq t_{0}.$ Hence,
\begin{align*}
\int^{\infty}_{0}&\|S(t)f\|\, dt \leq \|f\|_{\infty}\int^{t_{0}}_{0}\Biggl( \int_{K}\rho\bigl( x+iy+te^{-i\pi/4}\bigr)\, dx\, dy \Biggr)^{1/p}\, dt
\\ & + \|f\|_{\infty}\int_{t_{0}}^{\infty}\Biggl( \int_{K}\rho\bigl( x+iy+te^{-i\pi/4}\bigr)\, dx\, dy \Biggr)^{1/p}\, dt
\\ & \leq c_{0}\|f\|_{\infty}+c_{0}\|f\|_{\infty}\int_{t_{0}}^{\infty}\Biggl( \int_{K} e^{-\sqrt{t\sqrt{2}}} \, dx  \, dy\Biggr)^{1/p}\, dt<\infty,
\end{align*}
proving the claimed. The frequent hypercyclicity of $(T(t))_{t\in \Delta}$ follows by applying \cite[Proposition 2.1]{man-peris} and Proposition \ref{manic}. If $-\pi/4<\varphi \leq \pi/4,$
then there exists a finite constant $t_{\varphi}>0$ such that
\begin{align*}
t(\cos \varphi +\sin \varphi)\geq \sqrt{t(\cos \varphi -\sin \varphi)},\quad t\geq t_{\varphi},
\end{align*}
so that
\begin{align}\label{sel-man}
\rho(te^{i\varphi})=1,\quad t\geq t_{\varphi}.
\end{align}
Assume that the semigroup $(T(te^{i\varphi}))_{t\geq 0}$ is hypercyclic on $E$. Repeating literally the argumentation given in the first part of proof of \cite[Theorem 4.7]{fund}, with appealing to \cite[Lemma 4.2]{kuckanje-prim} in place of \cite[Lemma 4.2]{fund}, we obtain that there exists an increasing sequence $(t_{n})_{n\in {\mathbb N}}$ of positive reals tending to infinity such that $\lim_{n\rightarrow \infty}\rho(t_{n}e^{i\varphi})=0.$ This is in contradiction with \eqref{sel-man}.
\item[(ii)] (\cite[Example 4.17]{kuckanje-prim}) Let $\Delta:=\Delta(\pi/2),$ let $\rho(x+iy):=e^{-x+|y|}$ for $x+iy\in \Delta ,$ and let $E:=L^{p}_{\rho}(\Delta, {\mathbb K}).$ Applying
the Frequent Hypercyclic Criterion for $C_{0}$-semigroups, it can be easily seen that for each $\varphi \in (-\pi/4,\pi/4)$ the semigroup $(T(te^{i\varphi}))_{t\geq 0}$ is frequently hypercyclic on $E,$
so that the frequent hypercyclicity of $(T(t))_{t\in \Delta}$ follows again by applying \cite[Proposition 2.1]{man-peris} and Proposition \ref{manic}. If $|\varphi|\geq \pi/4,$ then the semigroup $(T(te^{i\varphi}))_{t\geq 0}$ cannot be hypercyclic, which can be established by the arguments given in the part (i) and the fact that there is no increasing sequence $(t_{n})_{n\in {\mathbb N}}$ of positive reals tending to infinity such that $\lim_{n\rightarrow \infty}\rho(t_{n}e^{i\varphi})=0.$
\item[(iii)] We will show that 
there exists a frequently hypercyclic semigroup with the index set $\Delta(\pi/4)$ and without any (frequently) hypercyclic single operator.

This example is based on \cite[Example 4.14]{kuckanje-prim}. Let $\Delta:=\Delta(\pi/4),$ $\zeta \geq 1,$ let
\begin{align*}
\rho(x+iy):=
\left\{
\begin{array}{l}
\Bigl(\frac{x+y+1}{x-y+1}\Bigr)^{2\zeta},\ x+y+1\geq \sqrt{x-y+1},\\
\Bigl(\frac{1}{x+y+1}\Bigr)^{2\zeta},\ x+y+1< \sqrt{x-y+1},
\end{array}
\right.
\end{align*}
and let $E:=L^{p}_{\rho}(\Delta, {\mathbb K}).$ Then $\rho(\cdot)$ is an admissible weight function and, arguing as in the afore-mentioned example, we can prove that $(T(t))_{t\in \Delta}$ is hypercyclic and that there is no number $t_{0}\in \Delta \setminus \{0\}$ such that
$T(t_{0})$ is hypercyclic.

 We  apply Theorem \ref{oma-niz} with $E_{0}$ being the space of continuous compactly supported functions, the sequence $T_{n}=T(t_{n})$ with
$$
t_{n}:=n+i\frac{-3-2n+\sqrt{8n+9}}
{2},\quad n\in {\mathbb N},
$$
and mappings  $S_{n} f: \Delta \rightarrow {\mathbb K}$ given by $(S_{n} f)(x):=f(x-t_{n} ),$ $x\in t_{n}+supp(f)$ and $(S_{n} f)(x):=0,$ otherwise ($f\in E_{0},$ $n\in {\mathbb N}$). In order to do that,
let us observe first that the condition (v) of Theorem \ref{oma-niz} holds wih a certain positive number $\delta >0$ and an increasing mapping $g(\cdot)$ with $g(n) \sim n\sqrt{2}$ as $n\rightarrow \infty.$
Fix now a function $f\in E_{0}$ with $supp(f)=K$ and observe first that for each $x+iy\in \Delta$ we have the following equivalence relation:
$x+y+1\geq \sqrt{x-y+1}$ iff $y\leq \frac{-3-2x+\sqrt{8x+9}}
{2}.$
Consider first the case $p=1.$ Then $(T_{t_{k}}S_{n+k}f)(x)=f(x-t_{k+n}+t_{k} ),$ $x\in t_{k+n}-t_{k}+supp(f)$ and $(T_{t_{k}}S_{n+k}f)(x)=0,$ otherwise ($k,\ n\in {\mathbb N}$). Furthermore, it can be easily seen that
\begin{align*}
\sum_{n=1}^{\infty}\bigl\|& T_{t_{k}}S_{n+k}f\bigr\|\leq \|f\|_{\infty}\sum_{n=1}^{\infty}\int_{t_{n+k}-t_{k}+K}\rho(x+iy)\, dx\, dy
\\& \leq \|f\|_{\infty}\sum_{n=1}^{\infty}\int_{t_{n+k}-t_{k}+K}\Biggl(\frac{1}{x+y+1}\Biggr)^{2\zeta}\, dx\, dy
\\&  = \|f\|_{\infty}\sum_{n=1}^{\infty}\int_{K}\Biggl(\frac{1}{x+y+1+\sqrt{8(n+k)+9}-\sqrt{8k+9}}\Biggr)^{2\zeta}\, dx\, dy
\\ & \leq \|f\|_{\infty}\sum_{n=1}^{\infty}m(K)2^{-2\zeta}n^{-\zeta},\quad k,\ n\in {\mathbb N},
\end{align*}
so that the series $\sum_{n=1}^{\infty}T(t_{k})S_{k+n }f$ converges absolutely, uniformly in $k\in {\mathbb N}.$ Similarly, the series $\sum_{n=1}^{k}T(t_{k})S_{k-n}f$ converges absolutely, uniformly in $k\in {\mathbb N},$
and the series $\sum_{n=1}^{\infty}S_{n }f$ converges absolutely. Since $ T_{t_{k}}S_{k}f=f,$ the claimed assertion follows.
Suppose now that $p>1.$ Then $E$ does not contain the space $c_{0}$ and it suffices to prove that for each functional $\psi \in L^{p'}_{\rho}(\Delta, {\mathbb K}),$ where $p'\in (1,\infty)$ satisfies $1/p+1/p'=1,$ we have that the series $\sum_{n=1}^{k}|\langle \psi,T(t_{k})S_{k-n}f\rangle |,$
$\sum_{n=1}^{\infty}|\langle \psi ,T(t_{k})S_{k+n }f \rangle |$ and $\sum_{n=1}^{\infty}|\langle \psi, S_{n }f \rangle |$ unconditionally converge, the first two of them uniformly in $k\in {\mathbb N}.$ This can be deduced as in the final part of proof of \cite[Proposition 3.3]{man-peris}. Observe that $(T(t))_{t\in \Delta}$ is $q$-frequently hypercyclic for any $q>1,$ provided that $\zeta=1,$ and that it is not clear whether $(T(t))_{t\in \Delta}$ is frequently hypercyclic in this case.

Finally, we would like to note, without going into full details, that the trick used above in combination with the analysis contained in \cite[Example 4.15]{kuckanje-prim} enables one to simply construct an example of translation semigroup $(T(t))_{t\in \Delta(\pi/4)}$ satisfying the following conditions: {\it for any ray $R,$ the semigroup  $(T(t))_{t\in R}$ is both topologically mixing and frequently hypercyclic but the whole semigroup is $(T(t))_{t\in \Delta(\pi/4)}$ is frequently hypercyclic and not topologically mixing (see \cite{kuckanje-prim} for the notion)}.
\item[(iv)]
 Let
$\Delta={\mathbb C}$ or $\Delta=\Delta(\alpha)$ for some $\alpha \in (0,\frac{\pi}{2}],$
let $m\in {{\mathbb N}_{0}},$ and let $C^{m}(\Delta, {\mathbb K})$
denote the vector space consisting of all functions $\varphi : \Delta
\rightarrow {\mathbb K}$ that are $m$ times continuously
differentiable in $\Delta^{\circ}$ and whose partial derivatives
$D^{\alpha} \varphi$ can be extended continuously throughout $\Delta
.$  Define $C^{\infty}(\Delta, {\mathbb
	K}):=\bigcap _{m\in {\mathbb N}}C^{m}(\Delta, {\mathbb K}).$
The Fr\' echet topology on $C^{m}(\Delta, {\mathbb K}),$ resp.
$C^{\infty}(\Delta, {\mathbb K}),$ is induced by the following system of increasing
seminorms:
$$
p_{n}(f):=\sup_{\tau \in \Delta_{n}}\sup_{|\alpha|\leq m}\bigl|D^{\alpha}
f (\tau)\bigr|,\ f\in C^{m}(\Delta, {\mathbb K}) \mbox{, resp. }
$$
$$
p_{n}(f):=\sup_{\tau \in \Delta_{n}}\sup_{|\alpha|\leq n}\bigl|D^{\alpha}
f(\tau)\bigr|, \ f\in C^{\infty}(\Delta, {\mathbb K}),\ n\in {\mathbb N}.
$$
Let $E:=C^{m}(\Delta, {\mathbb K})$ for some $m\in
{{\mathbb N}_{0}}$ or $E:=C^{\infty}(\Delta, {\mathbb K}).$
Based on  the consideration carried out in \cite[Example 3.1.29]{knjigaho},
we know that the translation semigroup $(T(t))_{t\in
	\Delta}$ is a locally equicontinuous $C_{0}$-semigroup in $E,$ and we already know that $(T(t))_{t\in
	\Delta}$ is both chaotic and topologically mixing. An easy application of Theorem \ref{oma} shows that $(T(t))_{t\in
	\Delta},$ $(T(t))_{t\in R}$ and any single operator $T(t_{0})$ are frequently hypercyclic, as well ($R\subseteq \Delta$ is a ray, $t_{0}\in \Delta \setminus \{0\}$).
 Then the translation semigroup $(T(t))_{t\in
\Delta}$ is a locally equicontinuous $C_{0}$-semigroup in $E,$ and we already know that $(T(t))_{t\in
\Delta}$ is both chaotic and topologically mixing. An easy application of Theorem \ref{oma} shows that $(T(t))_{t\in
\Delta},$ $(T(t))_{t\in R}$ and any single operator $T(t_{0})$ are frequently hypercyclic, as well ($R\subseteq \Delta$ is a ray, $t_{0}\in \Delta \setminus \{0\}$).
\end{itemize}

At the end of this section
we will prove the following extension of \cite[Proposition 2.1]{man-peris} for $f$-frequently hypercyclic semigroups on Fr\' echet spaces:

\begin{theorem}\label{f-freq-21.}
Let $(T(t))_{t\geq 0}$ be a $C_{0}$-semigroup on $E,$ let $f : [0,\infty) \rightarrow [1,\infty)$ be an increasing mapping, and let $m_{k}=f(k),$ $k\in {\mathbb N}_{0}.$
Suppose that
\begin{align}\label{baqart}
\liminf_{k\rightarrow \infty}\frac{m_{k}}{k}>0.
\end{align}
Then
the following statements are equivalent:
\begin{itemize}
\item[(i)] $(T(t))_{t\geq 0}$ is $f$-frequently hypercyclic;
\item[(ii)] for every $t>0$ the operator $T(t)$ is l-$(m_k)$-frequently hypercyclic;
\item[(iii)] there exists $t>0$ such that the operator $T(t)$ is l-$(m_k)$-frequently hypercyclic.
\end{itemize}
\end{theorem}

\begin{proof}
The implication (ii) $\Rightarrow$ (iii) is trivial, while the implication (iii) $\Rightarrow$ (i) follows from an application of Proposition \ref{manic}. Let a number $t_{0}>0$ be given and
let (i) hold. All that we need to show is that the operator $T(t_{0})$ is l-$(m_k)$-frequently hypercyclic. In order to see that, we will slightly modify the arguments given in the proofs of \cite[Lemma 3.1, Theorem 3.2]{cone2}. Without loss of generality, we may assume that $R(f)\subseteq {\mathbb N}.$ Let $x\in E$ be an $f$-frequently hypercyclic vector of $(T(t))_{t\geq 0}$. Fix $k\in {\mathbb N},$ $y\in E$ and two open balanced neighbourhoods of the origin $U$ and $U'$ in $E$
such that $U'+U'\subseteq U.$ Then, due to \cite[Theorem 2.3]{cone2}, $T(t_{0}j/k)x$ is a hypercyclic vector of $T(t_{0})$ for $j=0,1,2,\cdot \cdot \cdot,k-1,$ which means that there exist positive integers $n_{j}$ such that $T(n_{j}t_{0}+t_{0}(j/k))x\in y+U'$ for $j=0,1,2,\cdot \cdot \cdot,k-1.$
On the other hand, $(T(t))_{t\geq0}$ is strongly continuous so that there exists an open balanced neighbourhood of the origin $V$ in $E$ such that $T(s)(V)\subset U'$, for $s\leq N_0=\max\{n_j\, :\, j=0,1,2,\cdot \cdot \cdot,k-1\}+\max(1,t_{0})$.
By our assumption, there are numbers $C>0$ and $N_1\in{\mathbb N}$ such that $m(\{t\in[0,m_N]\, :\, T(t)x-x\in V\})\geq CN$, for every $N\geq N_1$.
For every $N\in{\mathbb N},$ we define $L:=\{t\in[0,m_N]\, :\, T_tx-x\in V\}$. Furthermore, for every $j\in {\mathbb N},$ we define
$$
I_j:=\bigcup\limits_{n\in {\mathbb N},m\in {{\mathbb N}_{0}}}\Biggl[t_{0}m_{n}+t_{0}s+\frac{j}{k}t_{0},t_{0}m_{n}+t_{0}s+\frac{j-1}{k}t_{0}\Biggr), \mbox{ } \ \ L_j:=I_j\cap L
$$
and the mappings $f_j:[0,\infty)\rightarrow[0,\infty)$ by $f_j(t):=t+n_{k-j-1}t_{0}+\frac{k-j-1}{k}t_{0}$.
Then it can be easily seen that $f_j(t)\in I_{k-1}$ and $T(f_j(t))x\subseteq U $ for every $t\in L_j.$
Hence, we have
\begin{align*}
m\Bigl(\bigl\{t\in [& 0,m_{N}+N_{0}+1]\, :\, T_tx-y\in U\mbox{ and }t\in I_{k-1}\bigr\}\Bigr)
\\ & \geq m\Biggl(\bigcup\limits_{j=0}^{k-1}f_j(L_j)\Biggr)\geq \sum\limits_{j=0}^{k-1}\frac{m\bigl(f_j(L_j)\bigr)}{k}
\\& =\sum\limits_{j=0}^{k-1}\frac{m(L_j)}{k}\geq \frac{m_{N}-m_{1}}{k}.
\end{align*}
Using \eqref{baqart} and this estimate, it readily follows that
$$
\liminf_{N\rightarrow \infty}\frac{m\bigl(\{ m_{N}\geq t\geq 0 : T(t)x-y\in V,\ t\in I_{k-1} \}\bigr)}{N}>0.
$$
Hence, there exist a positive constant $c>0$ and a strictly increasing sequence of positive integers $(N_{l})_{l\in {\mathbb N}}$ such that
$$
\frac{m\bigl(\{ m_{N_{l}}\geq t\geq 0 : T(t)x-y\in V,\ t\in I_{k-1} \}\bigr)}{N_{l}}>c,\quad l\in {\mathbb N}.
$$
For fixed $l\in {\mathbb N},$
there exists a number $t_{1}\in [0,m_{N_{l}}]$ such that $T(t)x-y\in V$ and $ t\in I_{k-1}.$ If  $t_{1}\in [t_{0}m_{n_{1}}+t_{0}s_{1}+\frac{k-1}{k}t_{0},t_{0}m_{n_{1}}+t_{0}s_{1}+t_{0})$ for some integers $n_{1}\in {\mathbb N}$ and $s_{1}\in {\mathbb N}_{0},$ then $t_{1}':=t_{0}(m_{n_{1}}+s_{1}+1)$ satisfies $|t_{1}-t_{1}|'\leq t_{0}/k$
and
$T(t_{1}')x\in y+U$ by the proof of \cite[Theorem 3.2]{cone2}. There exist integers $n_{2}\in {\mathbb N},\ s_{2}\in {\mathbb N}_{0}$ and a positive real number
$t_{2}\in [0,m_{N_{l}}]$ such that $m_{n_{2}}+s_{2}\neq m_{n_{1}}+s_{1},$
$T(t_{2})x-y\in V$ and $t_{2}\in [t_{0}m_{n_{2}}+t_{0}s_{2}+\frac{k-1}{k}t_{0},t_{0}m_{n_{2}}+t_{0}s_{2}+t_{0});$ otherwise, the Lebesgue measure of set $m(\{ m_{N_{l}}\geq t\geq 0 : T(t)x-y\in V,\ t\in I_{k-1} \})$ cannot be greater than $t_{0}/k.$ Set $t_{2}':=t_{0}(m_{n_{2}}+s_{2}+1).$ Then $t_{2}'\neq t_{1}'$ and $T(t_{2}')x\in y+U,$ as well. Proceeding in such a way, we may construct at least $\lfloor c'N_{l}\rfloor $ positive numbers $t_{1}',\cdot \cdot \cdot,t_{\lfloor c'N_{l}\rfloor}\in  [0,m_{N_{l}}]$ such that $T(t_{j}')x\in y+U$ for every $j=1,\cdot \cdot \cdot, \lfloor c'N_{l}\rfloor,$ where $c'>0$ is a certain positive constant. This simply implies that $x$ is an  l-$(m_k)$-frequently hypercyclic vector of $T(t_{0}),$ finishing the proof of theorem.
\end{proof}

It is worth noting that our result is optimal since the estimate \eqref{baqart} is necessary for the operator $T(1)$ to be l-$(m_k)$-frequently hypercyclic; it is also clear that \eqref{baqart} holds iff $\liminf_{x\rightarrow +\infty}\frac{f(x)}{x}>0,$ which is a necessary condition for $(T(t))_{t\geq 0}$ to be $f$-frequently hypercyclic. By the proofs of Proposition \ref{manic} and Theorem \ref{f-freq-21.}, we may conclude that, for a given element $x\in E,$ we have that $x$ is an $f$-frequently hypercyclic vector of $(T(t))_{t\geq 0}$ iff $x$ is an l-$(m_k)$-frequently hypercyclic vector of $T(t)$ for all $t>0$ iff $x$ is an l-$(m_k)$-frequently hypercyclic vector of $T(t)$ for some $t>0.$ Using this fact as well as the method proposed in the proof of Proposition \ref{markknop}, we may deduce the following:

\begin{proposition}\label{mjuzika}
Suppose that $(T(t))_{t\geq 0}$ is a hypercyclic $C_{0}$-semigroup on $E.$ Then there exist $x\in E$ and a strictly increasing sequence $(m_{n})_{n\in {\mathbb N}}$ of positive integers
such that for each open non-empty subset $V$ of $E$ and for each $t>0$ one has $\underline{d}_{m_{n}}(\{ k\in {\mathbb N} : T(kt) \in V\})=+\infty.$
\end{proposition}

We refer the reader to \cite{jaker}, \cite{kunze} and \cite{kazimir} for the notion of Pettis integrability in locally convex spaces. We close this section with the following observation:
\begin{remark} \label{nrem}
Since the assertions of  \cite[Theorem 4.3, Corollary 4.4]{man-peris} remain true in the setting of Fr\' echet spaces (see e.g. \cite[pp. 281-282, 291]{pettis}) and since the infinitesimal generator of a strongly continuous semigroup $(T(t))_{t\geq 0}$ on a Fr\'echet space is densely defined, the proof of \cite[Theorem 2.2]{man-peris} can be repated verbatim in order to see that
this statement holds true for $C_{0}$-semigroups in infinite-dimensional separable Fr\' echet spaces. The same holds with the statements of \cite[Proposition 2.7, Proposition 2.8]{man-peris}.
\end{remark}

\section[Generalized frequent hypercyclicity for translation semigroups...]{Generalized frequent hypercyclicity for translation semigroups and semigroups induced by semiflows}\label{srboljube}

In this section, we will use the same notion and notation as in Subsection \ref{srbo}. Define $\tilde{\rho} : {\mathbb C} \rightarrow [0,\infty)$ by $\tilde{\rho}(x):=\rho(x),$ $x\in \Delta$ and $\tilde{\rho}(x):=0,$ otherwise.


\begin{theorem}\label{dov}
Assume that $1\leq p<\infty,$ $E:=
L^{p}_{\rho}(\Delta, {\mathbb K}),$ $t_{0}\in \Delta' \setminus \{0\},$ $f : [0,\infty) \rightarrow [1,\infty)$ is an increasing mapping and $m_{k}=f(k),$ $k\in {\mathbb N}_{0}.$
Suppose that, for every $k\in {\mathbb N}$ and $\delta >0,$ there exist only finite number $c_{k,\delta}\in {\mathbb N}$ of tuples $(n_{1},n_{2})\in {\mathbb N}_{0}^{2}$
such that $n_{1}\neq n_{2},$
\begin{align}\label{lo-le}
\Bigl| \lfloor m_{k + n_{1}} \rfloor -\lfloor m_{k +n_{2}} \rfloor\Bigr| \leq \delta , \mbox{ and }\Bigl| \lfloor m_{k - n_{1}} \rfloor -\lfloor m_{k -n_{2}} \rfloor\Bigr| \leq \delta,\mbox{ provided }k\geq \max \bigl(n_{1},n_{2}\bigr),
\end{align}
as well as that the sequence $(c_{k,\delta})_{k\in {\mathbb N}}$ is bounded for every fixed number $\delta>0.$
Assume
that
the following conditions hold for any compact subset $K$ of $\Delta$:
\begin{itemize}
\item[(i)] The series $\sum_{n=1}^{k}\int_{K}\tilde{\rho}(x-t_{0}(\lfloor m_{k} \rfloor -\lfloor m_{k-n} \rfloor))\, dx$
converges unconditionally, uniformly in $k\in {\mathbb N}.$
\item[(ii)] The series $\sum_{n=1}^{\infty}\int_{ K}\tilde{\rho}(x-t_{0}(\lfloor m_{k+n} \rfloor -\lfloor m_{k} \rfloor))\, dx$ converges unconditionally, uniformly in $k\in {\mathbb N}.$
\item[(iii)] The series $\sum_{n=1}^{\infty}\int_{K}\rho(x+t_{0}\lfloor m_{n} \rfloor)\, dx$ converges unconditionally.
\end{itemize}
Then $(T(t))_{t\in \Delta'}$ is $f$-frequently hypercyclic, $(T(te^{i\arg(t_{0})}))_{t\geq 0}$ is $f$-frequently hypercyclic and the operator $T(t_{0})$ is l-$(m_{k})$-frequently hypercyclic.

In particular, if $f(t)=t^q+1$, $t\geq 0$, $q>1$ and $\tilde \rho(t)=(1+|t|)^{-s},$ $t\geq 0,$ where $qs>1,$ then
\eqref{lo-le} holds as well as \emph{(i)-(iii)}. Thus, $(T(t))_{t\in \Delta'}$ satisfies all the conclusions given above.
\end{theorem}

\begin{proof}
	Without loss of generality, we may assume that $\Delta'=\Delta.$

First, 	we will prove the particular case.
Assume that  the numbers $k\in {\mathbb N}$ and $\delta>0$ are given as well as that $| \lfloor m_{k + n_{1}} \rfloor -\lfloor m_{k +n_{2}} \rfloor| \leq \delta$ for some positive integers $n_{1},\ n_{2}\in {\mathbb N}$ such that $n_{1}< n_{2}.$ Then $|(k+n_{1})^{q}-(k+n_{2})^{q}|\leq \delta +2$ and since, by the mean value theorem, $|(k+n_{1})^{q}-(k+n_{2})^{q}|\geq q|n_{1}-n_{2}|(k+n_{1})^{q-1},$ we get $q(k+n_{1})^{q-1}\leq \delta +2,$ which immediately implies
$$
n_{1}\leq \Bigl( \frac{\delta+2}{q} \Bigr)^{1/(q-1)}-k\ \mbox{ and }\  n_{2}\leq \Bigl( \delta +2+\bigl[(\delta+2)/q\bigr]^{q/(q-1)} \Bigr)^{1/q}-k.
$$
Hence, there exist only finite number $c_{k,\delta}'\in {\mathbb N}$ of tuples $(n_{1},n_{2})\in {\mathbb N}_{0}^{2}$
such that $n_{1}\neq n_{2},$ the first inequality in \eqref{lo-le} holds and the sequence $(c_{k,\delta}')_{k\in {\mathbb N}}$ is bounded for every fixed number $\delta>0.$ Similarly, the assumption $| \lfloor m_{k - n_{1}} \rfloor -\lfloor m_{k -n_{2}} \rfloor| \leq \delta$ for $k\geq \max (n_{1},n_{2})$ implies
\begin{align}\label{lako-je}
[n_{1},n_{2}]\subseteq \Biggl[ k-\Bigl( \delta +2+ [(\delta+2)/q]^{1/q}\Bigr)^{q/(q-1)}, k-\Bigl( \frac{\delta +2}{q}\Bigr)^{1/(q-1)}\Biggr],
\end{align}
on account of which
there exist only finite number $c_{k,\delta}''\in {\mathbb N}$ of tuples $(n_{1},n_{2})\in {\mathbb N}_{0}^{2}$
such that $n_{1}\neq n_{2},$ the second inequality in \eqref{lo-le} holds and the sequence $(c_{k,\delta}'')_{k\in {\mathbb N}}$ is bounded for every fixed number $\delta>0$ ($n_{1},\ n_{2}$ are integers and the interval appearing on the right hand side of \eqref{lako-je} contains only finite number of integers, which is independent of $k\in {\mathbb N}$).
Suppose now that $qs>1.$
Since $t_0\neq 0,$ elementary inequalities give
$$\frac{1}{(1+|x-t_0(\lfloor k^q\rfloor-\lfloor (k-n)^q \rfloor )|)^s}\sim \frac{1}{|t_0|^s(\lfloor k^q\rfloor-\lfloor (k-n)^q\rfloor)^s}\leq \frac{1}{|t_0|^sn^{qs}}
$$
(since $(a+b)^q\geq a^q+b^q$), uniformly in $x\in K.$ Thus, (i) holds; the proofs of (ii) and (iii) follow in the same way.
Thus, in this case, $(T(t))_{t\in \Delta'}$ satisfies all the assumptions of the main part of the theorem which will be proved now.

The main result is a  consequence of Theorem \ref{oma}, with $E_{0}$ being the set consisting of all continuous functions with compact support and the mapping $S_{\lfloor m_{n} \rfloor} f: \Delta \rightarrow {\mathbb K}$ given by $(S_{\lfloor m_{n} \rfloor} f)(x):=f(x-t_{0}\lfloor m_{n} \rfloor),$ $x\in t_{0}\lfloor m_{n} \rfloor+supp(f)$ and $(S_{\lfloor m_{n} \rfloor} f)(x):=0,$ otherwise ($f\in E_{0},$ $n\in {\mathbb N}$).
Then
$T(t_{0}\lfloor m_{n} \rfloor)S_{\lfloor m_{n} \rfloor}f=f$ for all $f\in E_{0},\ n\in {\mathbb N}$ and the conditions (i)-(iii) enable one to see that the requirements of Theorem \ref{oma} are satisfied. Strictly speaking, let a function $f\in E_{0}$ be given and let $supp(f)=K\subseteq \Delta.$ Set
$\delta :=\max\{2|t| : t\in K\}/|t_{0}|.$
To see that the series $\sum_{n=1}^{k}T(t_{0}\lfloor m_{k}\rfloor)S_{\lfloor m_{k-n}\rfloor}f$ converges unconditionally, uniformly in $k\in {\mathbb N},$ observe first that
\begin{align*}
\Biggl\|\sum_{n=1}^{k}& T(t_{0}\lfloor m_{k}\rfloor)S_{\lfloor m_{k-n}\rfloor}f\Biggl\|
\\ & =\Biggl( \int_{\Delta}\Biggl|\sum_{n=1}^{k}f\bigl(x+t_{0}(\lfloor m_{k} \rfloor -\lfloor m_{k-n} \rfloor)\bigr)\chi_{K-t_{0}(\lfloor m_{k} \rfloor -\lfloor m_{k-n} \rfloor)}(x) \Biggr|^{p}\rho(x)\, dx\Biggr)^{1/p}.
\end{align*}
For any fixed number $x\in \Delta,$ the sum
$$
\sum_{n=1}^{k}f\bigl(x+t_{0}(\lfloor m_{k} \rfloor -\lfloor m_{k-n} \rfloor)\bigr)\chi_{K-t_{0}(\lfloor m_{k} \rfloor -\lfloor m_{k-n} \rfloor)}(x)
$$
is consisted of at most $c_{k,\delta}$ addends since the assumption $x\in K-t_{0}(\lfloor m_{k} \rfloor -\lfloor m_{k-n_{1}} \rfloor)$ and $x\in K-t_{0}(\lfloor m_{k} \rfloor -\lfloor m_{k-n_{2}} \rfloor)$ for some integers $n_{1},\ n_{2} \in [1,k]$ implies $| \lfloor m_{k+n_{1}} \rfloor -\lfloor m_{k+n_{2}} \rfloor| \leq \delta .$
Hence,
\begin{align*}
& \Biggl\|\sum_{n=1}^{k} T(t_{0}\lfloor m_{k}\rfloor)S_{\lfloor m_{k-n}\rfloor}f\Biggl\|
\\ & \leq c_{k,\delta}^{(p-1)/p}\Biggl( \int_{\Delta}\sum_{n=1}^{k}\bigl|f\bigl(x+t_{0}(\lfloor m_{k} \rfloor -\lfloor m_{k-n} \rfloor)\bigr) \chi_{K-t_{0}(\lfloor m_{k} \rfloor -\lfloor m_{k-n} \rfloor)}(x) \bigr|^{p}\tilde{\rho}(x)\, dx\Biggr)^{1/p}
\\ & =c_{k,\delta}^{(p-1)/p}\Biggl(\sum_{n=1}^{k} \int_{K \cap (K +t_{0}(\lfloor m_{k} \rfloor -\lfloor m_{k-n} \rfloor))}\bigl|f(x)\bigr|^{p}\tilde{\rho}\bigl(x-t_{0}(\lfloor m_{k} \rfloor -\lfloor m_{k-n} \rfloor)\bigr)\, dx\Biggr)^{1/p}
\\ & \leq c_{k,\delta}^{(p-1)/p}\Biggl(\sum_{n=1}^{k} \int_{K}\bigl|f(x)\bigr|^{p}\tilde{\rho}\bigl(x-t_{0}(\lfloor m_{k} \rfloor -\lfloor m_{k-n} \rfloor)\bigr)\, dx\Biggr)^{1/p}
\\ & \leq  c_{k,\delta}^{(p-1)/p}\|f\|_{\infty}\Biggl(\sum_{n=1}^{k}\int_{K}\tilde{\rho} \bigl(x-t_{0}(\lfloor m_{k} \rfloor -\lfloor m_{k-n} \rfloor)\bigr)\, dx\Biggr)^{1/p}.
\end{align*}
Similarly, the series $\sum_{n=1}^{\infty}T(t_{0}\lfloor m_{k}\rfloor)S_{\lfloor m_{k+n} \rfloor}f$ converges unconditionally, uniformly in $k\in {\mathbb N},$ due to (ii) and the monotone convergence theorem:
\begin{align*}
\Biggl\|&\sum_{n=1}^{\infty}T(t_{0}\lfloor m_{k}\rfloor)S_{\lfloor m_{k+n} \rfloor}f\Biggl\|
\\ &
\leq c_{k,\delta}^{(p-1)/p}\Biggl( \int_{\Delta}\sum_{n=1}^{\infty}\bigl|f\bigl(x+t_{0}(\lfloor m_{k+n} \rfloor -\lfloor m_{k} \rfloor)\bigr) \chi_{K-t_{0}(\lfloor m_{k+n} \rfloor -\lfloor m_{k} \rfloor)}(x) \bigr|^{p}\tilde{\rho}(x)\, dx\Biggr)^{1/p}
\\ & =c_{k,\delta}^{(p-1)/p}\Biggl( \sum_{n=1}^{\infty} \int_{\Delta}\bigl|f\bigl(x+t_{0}(\lfloor m_{k+n} \rfloor -\lfloor m_{k} \rfloor)\bigr) \chi_{K-t_{0}(\lfloor m_{k+n} \rfloor -\lfloor m_{k} \rfloor)}(x) \bigr|^{p}\tilde{\rho}(x)\, dx\Biggr)^{1/p}
\\ &
\leq c_{k,\delta}^{(p-1)/p}\Biggl(\sum_{n=1}^{\infty} \int_{K}\bigl|f(x)\bigr|^{p}\tilde{\rho} \bigl(x-t_{0}(\lfloor m_{k+n} \rfloor -\lfloor m_{k} \rfloor)\bigr)\, dx\Biggr)^{1/p}
\\ & \leq  c_{k,\delta}^{(p-1)/p}\|f\|_{\infty}\Biggl(\sum_{n=1}^{\infty}\int_{K}\tilde{\rho} \bigl(x-t_{0}(\lfloor m_{k+n} \rfloor -\lfloor m_{k} \rfloor)\bigr)\, dx\Biggr)^{1/p}.
\end{align*}
Since \eqref{lo-le} holds, for any compact set $K$ of $\Delta$ there exists a finite number $n_{K}\in {\mathbb N}$ such that any point $x\in \Delta$ is contained at most $n_{K}$ sets of the form $t_{0}\lfloor m_{n} \rfloor +K,$ where $n\in {\mathbb N}.$ Using this condition and
the condition (iii), we can similarly verify that the series $\sum_{n=1}^{\infty}S_{\lfloor m_{n} \rfloor}f$ converges unconditionally.
\end{proof}

\begin{remark}\label{hnme}
	The condition \eqref{lo-le} does not hold for  frequent hypercyclicity ($q=1$).
\end{remark}

\begin{remark}\label{hnme-f}
	Let $f : [0,\infty) \rightarrow [1,\infty)$ be two times continuously differentiable function satisfying additionally that $f^{\prime \prime}(x)\geq 0$ for all $x\geq 0,$ as well as $\lim_{x\rightarrow +\infty}f(x)=\lim_{x\rightarrow +\infty}f^{\prime}(x)=+\infty .$
	Arguing in the same manner, we can prove that
	the condition \eqref{lo-le} holds for $f$-frequent hypercyclicity, i.e., for any numbers $k\in {\mathbb N}$ and $\delta>0$ given in advance,
	there exist only finite number $c_{k,\delta}\in {\mathbb N}$ of tuples $(n_{1},n_{2})\in {\mathbb N}_{0}^{2}$
	such that $n_{1}\neq n_{2},$ \eqref{lo-le} holds and the sequence $(c_{k,\delta})_{k\in {\mathbb N}}$ is bounded for every fixed number $\delta>0.$
	By a careful analysis, one can find a corresponding $\tilde {\rho}(\cdot)$ so that conditions (i)-(iii) automatically hold. We will not consider these cases.
\end{remark}

For translation semigroups acting on $
C_{0,\rho}(\Delta, {\mathbb K}),$ we have the following result:

\begin{theorem}\label{dovs}
Assume that $E:=
C_{0,\rho}(\Delta, {\mathbb K}),$ $f : [0,\infty) \rightarrow [1,\infty)$ is an increasing mapping and $m_{k}=f(k),$ $k\in {\mathbb N}_{0}.$ Suppose that, for every $k\in {\mathbb N}$ and $\delta >0,$ there exist only finite number $c_{k,\delta}\in {\mathbb N}$ of tuples $(n_{1},n_{2})\in {\mathbb N}_{0}^{2}$
such that $n_{1}\neq n_{2},$
\eqref{lo-le} holds
as well as that the sequence $(c_{k,\delta})_{k\in {\mathbb N}}$ is bounded for every fixed number $\delta>0.$
Suppose that there is a number $t_{0} \in \Delta' \setminus \{0\}$ such that
the following conditions hold for any compact subset $K$ of $\Delta$:
\begin{itemize}
\item[(i)] For any $\epsilon>0$ there is an integer
$N\in {\mathbb N}$ such that for any finite set $F\subseteq [N,\infty) \cap {\mathbb N}$ and for any $k\in {\mathbb N}$ we have
$$
\sup_{n\in F,k\geq n, x\in K}\tilde{\rho} \bigl(x-t_{0}(\lfloor m_{k} \rfloor -\lfloor m_{k-n} \rfloor)\bigr) \in L(0,\epsilon).
$$
\item[(ii)] For any $\epsilon>0$ there is an integer
$N\in {\mathbb N}$ such that for any finite set $F\subseteq [N,\infty) \cap {\mathbb N}$ and for any $k\in {\mathbb N}$ we have
$$
\sup_{n\in F,x\in K}\tilde{\rho} \bigl(x-t_{0}(\lfloor m_{k+n} \rfloor -\lfloor m_{k} \rfloor)\bigr) \in L(0,\epsilon).
$$
\item[(iii)] For any $\epsilon>0$ there is an integer
$N\in {\mathbb N}$ such that for any finite set $F\subseteq [N,\infty) \cap {\mathbb N}$ we have
$$
\sup_{n\in F,x\in K}\rho \bigl(x+t_{0}\lfloor m_{n} \rfloor \bigr) \in L(0,\epsilon).
$$
\end{itemize}
Then $(T(t))_{t\in \Delta'}$ is $f$-frequently hypercyclic, $(T(te^{i\arg(t_{0})}))_{t\geq 0}$ is $f$-frequently hypercyclic and the operator $T(t_{0})$ is l-$(m_{k})$-frequently hypercyclic.

In particular, if $f(t)=t^q+1$, $t\geq 0$, $q>1$ and $\tilde \rho(t)=(1+|t|)^{-s},$ $t\geq 0,$ where $qs>1,$ then
\eqref{lo-le} holds as well as \emph{(i)-(iii)}. Thus, $(T(t))_{t\in \Delta'}$ satisfies all the conclusions given above.
\end{theorem}

\begin{proof}
The proof is similar to that of Theorem \ref{dov} and the basic differences are given below (the particular case is the same). Again,
we may assume that $\Delta'=\Delta.$
We can apply Theorem \ref{oma}, with $E_{0}$ being the set consisting of all continuous functions with compact support and the mappings $S_{\lfloor m_{n} \rfloor} \cdot$
defined as in the proof of Theorem \ref{dov} ($n\in {\mathbb N}$).
Then
$T(t_{0}\lfloor m_{n} \rfloor)S_{\lfloor m_{n} \rfloor}f=f$ for all $f\in E_{0},\ n\in {\mathbb N}$ and (i)-(iii) implies the validity of conditions necessary for applying Theorem \ref{oma}. Strictly speaking, suppose that a function $f\in E_{0}$ is given and $supp(f)=K\subseteq \Delta.$ Set
$\delta :=\max\{2|t| : t\in K\}/|t_{0}|.$
To see that the series $\sum_{n=1}^{k}T(t_{0}\lfloor m_{k}\rfloor)S_{\lfloor m_{k-n}\rfloor}f$ converges unconditionally, uniformly in $k\in {\mathbb N},$ observe first that for each finite set $F\subseteq {\mathbb N}$ and each fixed number $x\in \Delta,$ the sum
$$
\sum_{1\leq n \leq k,n\in F}f\bigl(x+t_{0}(\lfloor m_{k} \rfloor -\lfloor m_{k-n} \rfloor)\bigr)\chi_{K-t_{0}(\lfloor m_{k} \rfloor -\lfloor m_{k-n} \rfloor)}(x)
$$
consists at most $c_{k,\delta}$ addends and therefore
\begin{align*}
\Biggl|\sum_{1\leq n \leq k,n\in F}& f\bigl(x+t_{0}(\lfloor m_{k} \rfloor -\lfloor m_{k-n} \rfloor)\bigr)\chi_{K-t_{0}(\lfloor m_{k} \rfloor -\lfloor m_{k-n} \rfloor)}(x)\Biggr|
\\ & \leq c_{k,\delta}\sup_{1\leq n\leq k,n\in F}\Bigl| f\bigl(x+t_{0}(\lfloor m_{k} \rfloor -\lfloor m_{k-n} \rfloor)\bigr)\Bigr|.
\end{align*}
This implies
\begin{align*}
&  \Biggl\|\sum_{1\leq n \leq k,n\in F}T(t_{0}\lfloor m_{k}\rfloor)S_{\lfloor m_{k-n}\rfloor}f\Biggl\|
\\ & =\sup_{x\in \Delta}\Biggl|\sum_{1\leq k\leq n,n\in F}f\bigl(x+t_{0}(\lfloor m_{k} \rfloor -\lfloor m_{k-n} \rfloor)\bigr)\chi_{K-t_{0}(\lfloor m_{k} \rfloor -\lfloor m_{k-n} \rfloor)}(x) \Biggr|\tilde{\rho}(x)
\\ & \leq c_{k,\delta}\sup_{1\leq n\leq k,n\in F,x\in \Delta}\Bigl| f\bigl(x+t_{0}(\lfloor m_{k} \rfloor -\lfloor m_{k-n} \rfloor)\bigr)\Bigr|\tilde{\rho}(x)
\\ &\leq c_{k,\delta}\sup_{1\leq n\leq k,n\in F,x\in K}| f(x)|\tilde{\rho} \bigl(x-t_{0}(\lfloor m_{k} \rfloor -\lfloor m_{k-n} \rfloor)\bigr)
\\ & \leq  c_{k,\delta}\|f\|_{\infty}\sup_{1\leq n\leq k,n\in F,x\in K}\tilde{\rho} \bigl(x-t_{0}(\lfloor m_{k} \rfloor -\lfloor m_{k-n} \rfloor)\bigr).
\end{align*}
Using this estimate and condition (i), we get that the series $\sum_{n=1}^{k}T(t_{0}\lfloor m_{k}\rfloor)S_{\lfloor m_{k-n}\rfloor}f$ converges unconditionally, uniformly in $k\in {\mathbb N}.$
Similarly, we can prove that the series $\sum_{n=1}^{\infty}T(t_{0}\lfloor m_{k}\rfloor)S_{\lfloor m_{k+n} \rfloor}f$ converges unconditionally, uniformly in $k\in {\mathbb N},$ and that the series $\sum_{n=1}^{\infty}S_{\lfloor m_{n} \rfloor}f$ converges unconditionally.
\end{proof}

We can similarly prove the following results continuing our previous analysis from Theorem \ref{oma-niz} (it is only worth noting that we need the condition \eqref{modric} below for proving that for any compact set $K$ of $\Delta$ there exists a finite number $n_{K}\in {\mathbb N}$ such that any point $x\in \Delta$ is contained at most $n_{K}$ sets of the form $t_{\lfloor m_{n} \rfloor}+K,$ where $n\in {\mathbb N}$):

\begin{theorem}\label{dov-niz}
Assume that $1\leq p<\infty,$ $E:=
L^{p}_{\rho}(\Delta, {\mathbb K}),$ $(t_{n})_{n\in {\mathbb N}}$ is a sequence in $\Delta' \setminus \{0\},$
$\delta'>0,$ \eqref{krke-rke} holds with the number $\delta$ replaced therein with the number $\delta',$
there exists an increasing mapping $g : [1,\infty) \rightarrow [1,\infty)$ such that $|t_{n}|\leq g(n)$ for all $n\in {\mathbb N},$
$f : [0,\infty) \rightarrow [1,\infty)$ is an increasing mapping and $m_{k}=f(k),$ $k\in {\mathbb N}_{0}.$
Suppose that, for every $k\in {\mathbb N}$ and $\delta >0,$ there exist only finite number $c_{k,\delta}\in {\mathbb N}$ of tuples $(n_{1},n_{2})\in {\mathbb N}_{0}^{2}$
such that $n_{1}\neq n_{2},$ \eqref{lo-le} holds
and the sequence $(c_{k,\delta})_{k\in {\mathbb N}}$ is bounded for every fixed number $\delta>0.$
Assume that
\begin{align}\label{modric}
\lim_{n\rightarrow \infty}\bigl| t_{\lfloor m_{n+1} \rfloor}-t_{\lfloor m_{n} \rfloor}\bigr|=+\infty
\end{align}
as well as
the following conditions hold for any compact subset $K$ of $\Delta$:
\begin{itemize}
\item[(i)] The sequence $\sum_{n=1}^{k}\int_{K}\tilde{\rho}(x-t_{\lfloor m_{k} \rfloor}+t_{\lfloor m_{k-n} \rfloor})\, dx,$ $ k\in\mathbb N$ converges unconditionally, uniformly in $k\in {\mathbb N}.$
\item[(ii)] The series $\sum_{n=1}^{\infty}\int_{ K}\tilde{\rho}(x-t_{\lfloor m_{k+n} \rfloor}+ t_{\lfloor m_{k} \rfloor})\, dx$ converges unconditionally, uniformly in $k\in {\mathbb N}.$
\item[(iii)] The series $\sum_{n=1}^{\infty}\int_{K}\rho(x+t_{\lfloor m_{n} \rfloor})\, dx$ converges unconditionally.
\end{itemize}
Then $(T(t))_{t\in \Delta'}$ is $(g\circ f)$-frequently hypercyclic and the sequence of operators
$(T_{n}:=T(t_{n}))_{n\in {\mathbb N}}$ is l-$(m_{k})$-frequently hypercyclic.
\end{theorem}

\begin{theorem}\label{dovs-niz}
Assume that $E:=
C_{0,\rho}(\Delta, {\mathbb K}),$ $f : [0,\infty) \rightarrow [1,\infty)$ is an increasing mapping and $m_{k}=f(k),$ $k\in {\mathbb N}_{0}.$ Suppose that, for every $k\in {\mathbb N}$ and $\delta >0,$ there exist only finite number $c_{k,\delta}\in {\mathbb N}$ of tuples $(n_{1},n_{2})\in {\mathbb N}_{0}^{2}$
such that $n_{1}\neq n_{2},$ \eqref{lo-le} holds
and the sequence $(c_{k,\delta})_{k\in {\mathbb N}}$ is bounded for every fixed number $\delta>0.$
Suppose, further, that \eqref{modric} holds as well as that there are a  sequence $(t_{n})_{n\in {\mathbb N}}$ in $\Delta' \setminus \{0\}$ and a number
$\delta'>0$ such that \eqref{krke-rke} holds with the number $\delta$ replaced therein with the number $\delta'.$ Let $g : [1,\infty) \rightarrow [1,\infty)$ be an
increasing mapping such that $|t_{n}|\leq g(n)$ for all $n\in {\mathbb N},$
as well as that
the following conditions hold for any compact subset $K$ of $\Delta$:
\begin{itemize}
\item[(i)] For any $\epsilon>0$ there is an integer
$N\in {\mathbb N}$ such that for any finite set $F\subseteq [N,\infty) \cap {\mathbb N}$ and for any $k\in {\mathbb N}$ we have
$$
\sup_{n\in F,k\geq n, x\in K}\tilde{\rho} \bigl(x-t_{\lfloor m_{k} \rfloor}+t_{\lfloor m_{k-n} \rfloor}\bigr) \in L(0,\epsilon).
$$
\item[(ii)] For any $\epsilon>0$ there is an integer
$N\in {\mathbb N}$ such that for any finite set $F\subseteq [N,\infty) \cap {\mathbb N}$ and for any $k\in {\mathbb N}$ we have
$$
\sup_{n\in F,x\in K}\tilde{\rho} \bigl(x-t_{\lfloor m_{k+n} \rfloor}+ t_{\lfloor m_{k} \rfloor}\bigr) \in L(0,\epsilon).
$$
\item[(iii)] For any $\epsilon>0$ there is an integer
$N\in {\mathbb N}$ such that for any finite set $F\subseteq [N,\infty) \cap {\mathbb N}$ we have
$$
\sup_{n\in F,x\in K}\rho \bigl(x+t_{\lfloor m_{n} \rfloor}\bigr) \in L(0,\epsilon).
$$
\end{itemize}
Then $(T(t))_{t\in \Delta'}$ is $(g\circ f)$-frequently hypercyclic and the sequence of operators
$(T_{n}:=T(t_{n}))_{n\in {\mathbb N}}$ is l-$(m_{k})$-frequently hypercyclic.
\end{theorem}

Keeping in mind \cite[Lemma 4.2]{kuckanje-prim}, it is almost straightforward to extend the assertions of \cite[Proposition 3.6, Proposition 3.8]{man-peris} for $f$-frequently hypercyclic $C_{0}$-semigroups on complex sectors which do have at least one single  l-$(m_{n})$-hypercyclic operator (the situation is not so clear for  $C_{0}$-semigroups defined on complex sector $\Delta \neq [0,\infty)$ without any single  l-$(m_{n})$-hypercyclic operator):

\begin{proposition}\label{backspace}
\begin{itemize}
\item[(i)] Assume that $1\leq p<\infty,$ $E:=
L^{p}_{\rho}(\Delta, {\mathbb K}),$ $t_{0}\in \Delta' \setminus \{0\},$ $f : [0,\infty) \rightarrow [1,\infty)$ is an increasing mapping and $m_{k}=f(k),$ $k\in {\mathbb N}_{0}.$ Let the operator $T(t_{0})$ be l-$(m_{n})$-hypercyclic. Then for each $\epsilon>0$ there exists a strictly increasing sequence $(n_{k})_{k\in {\mathbb N}}$ in ${\mathbb N}$ such that the set $\{n_{k} : k\in {\mathbb N}\}$ has a positive
lower $(m_{n})$-density and that, for every $i\in {\mathbb N},$
$$
\sum_{k>i}\rho \Bigl( \bigl(n_{k}-n_{i}\bigr)e^{i\arg(t_{0})}\Bigr)<\epsilon.
$$
In particular, if $f(t)=t+1$ for all $t\geq 0,$ then the mapping $t\mapsto \rho(te^{i\arg(t_{0})}),$ $t\geq 0$ is bounded.
\item[(ii)] Assume that $E:=
C_{0,\rho}(\Delta, {\mathbb K}),$ $t_{0}\in \Delta' \setminus \{0\},$ $f : [0,\infty) \rightarrow [1,\infty)$ is an increasing mapping and $m_{k}=f(k),$ $k\in {\mathbb N}_{0}.$ Let the operator $T(t_{0})$ be l-$(m_{n})$-hypercyclic. Then for each $\epsilon>0$ there exists a strictly increasing sequence $(n_{k})_{k\in {\mathbb N}}$ in ${\mathbb N}$ such that the set $\{n_{k} : k\in {\mathbb N}\}$ has a positive
lower $(m_{n})$-density and that, for every $i\in {\mathbb N},$ we have
$\rho ( (n_{k}-n_{i})e^{i\arg(t_{0})})<\epsilon.
$ In particular, if $f(t)=t+1$ for all $t\geq 0,$ then the mapping $t\mapsto \rho(te^{i\arg(t_{0})}),$ $t\geq 0$ is bounded.
\end{itemize}
\end{proposition}

\subsection{Examples II}
\label{pogreska}
\begin{itemize}

\item[(i)] Let $\Delta =[0,\infty)$ or $\Delta={\mathbb R},$ and let $\rho(x):=(1+|x|)^{-1},$ $x\in \Delta.$ Then $\rho(\cdot)$ is an
admissible weight function and, due to \cite[Theorem 3.8, Theorem 3.9]{man-peris}, the translation semigroup $(T(t))_{t\geq 0}$ is not frequently hypercyclic on $
L^{p}_{\rho}(\Delta, {\mathbb K}),$ for any $p\geq 1.$ Suppose now that $q>1,$ $1\leq p <\infty$ and $f(t):=t^{q}+1,$ $t\geq 0.$ Then we can apply Theorem \ref{dov} in order to see that $(T(t))_{t\geq 0}$ is $q$-frequently hypercyclic on $
L^{p}_{\rho}(\Delta, {\mathbb K}).$ Further on, due to \cite[Proposition 2.7]{man-peris}, any single operator $T(t_{0})$ of a $C_{0}$-semigroup $(T(t))_{t\geq 0}$ satisfying the requirements of Frequent Hypercyclic Criterion for $C_{0}$-semigroups needs to be chaotic ($t_{0}>0$). Here, we would like to point out that, due to \cite[Corollary 4.8]{hasan-emamirad}, the translation $q$-frequently hypercyclic semigroup $(T(t))_{t\geq 0}$ constructed in this way, with $\Delta =[0,\infty)$, is not chaotic and that, for every $t_{0}>0,$ the single operator $T(t_{0})$ is not chaotic.
\item[(ii)] (\cite[Example
1]{takeo}; see also \cite[Example 3.7]{man-peris} and \cite[Example 3.1.28(ii)]{knjigaho}) Let $p\in [1,\infty),$ $\Delta =\Delta':=[0,\infty)$ and $\rho(x):=e^{-(x+1)\cos( \ln (x+1))+1},$ $x\geq 0.$ Assume first that $E:=L^{p}_{\rho}([0,\infty ),{\mathbb K}).$ Then the translation semigroup $(T(t))_{t\geq 0}$ is hypercyclic because $\liminf_{x\rightarrow \infty}\rho(x)=0,$ and not chaotic because $\int^{\infty}_{0}\rho(x)\, dx=+\infty;$ by \cite[Theorem 3]{man-marina}, it follows that $(T(t))_{t\geq 0}$ is not frequently hypercyclic. Let $\zeta \sim \pi/2+.$ We will prove that $(T(t))_{t\geq 0}$ is $f$-frequently hypercyclic with $f(x):=e^{(2x+1)\pi +\zeta},$ $x\geq 0.$ By Remark \ref{hnme-f}, we have that for given numbers
$k\in {\mathbb N}$ and $\delta >0,$ there exist only finite number $c_{k,\delta}\in {\mathbb N}$ of tuples $(n_{1},n_{2})\in {\mathbb N}_{0}^{2}$
such that $n_{1}\neq n_{2},$ \eqref{lo-le} holds
as well as that the sequence $(c_{k,\delta})_{k\in {\mathbb N}}$ is bounded for every fixed number $\delta>0.$ Let $t_{0}=1$ and let $K=[a,b]\subseteq [0,\infty)$ be a compact set. First of all, it can be easily seen that the series $\sum_{n=1}^{k}\int_{K}\tilde{\rho}(x-(\lfloor m_{k} \rfloor -\lfloor m_{k-n} \rfloor))\, dx$ and $\sum_{n=1}^{\infty}\int_{ K}\tilde{\rho}(x-(\lfloor m_{k+n} \rfloor -\lfloor m_{k} \rfloor))\, dx$ converge unconditionally, uniformly in $k\in {\mathbb N},$ because the sums are finite and equal zero for $k\geq k_{0},$ where $k_{0}\in {\mathbb N}$ depends only on $K.$ The series $\sum_{n=1}^{\infty}\int_{a}^{b}\rho(x+\lfloor m_{n} \rfloor)\, dx$ converges absolutely
because there exists an integer $n_{0}(K) \in {\mathbb N}$ such that $[\ln(a+1+\lfloor e^{\zeta+(2n+1)\pi} \rfloor) ,\ln(b+1+\lfloor e^{\zeta+(2n+1)\pi} \rfloor ) ]\subseteq  [\zeta +(2n+1) \pi,\pi-\zeta+(2n+3)\pi]$ for all $n\in {\mathbb N},$ which clearly implies that there exists a finite constant $c_{\zeta ,K}>0$ such that
\begin{align*}
\sum_{n=1}^{\infty}\int_{a}^{b}& \rho(x+\lfloor m_{n} \rfloor)\, dx
=\sum_{n=1}^{\infty}\int_{a+\lfloor m_{n} \rfloor}^{b+\lfloor m_{n} \rfloor} \rho(x)\, dx
\\ & = \sum_{n=1}^{\infty}\int_{a+\lfloor e^{\zeta +(2n+1)\pi  }\rfloor}^{b+\lfloor e^{\zeta +(2n+1)\pi  }\rfloor} \rho(x)\, dx \leq c_{\zeta ,K}\sum_{n=1}^{\infty}e^{-(2n+1) \pi}.
\end{align*}
Hence, an application of Theorem \ref{dov} shows that $(T(t))_{t\geq 0}$ is $f$-frequently hypercyclic, as claimed. Suppose now that
$E:=C_{0,\rho}([0,\infty ) ,{\mathbb K}).$ Then $(T(t))_{t\geq 0}$ is hypercyclic, not chaotic and moreover, $(T(t))_{t\geq 0}$ cannot be frequently hypercyclic by \cite[Proposition 3.8]{man-peris}. Arguing as above, we may conclude that $(T(t))_{t\geq 0}$ is $f$-frequently hypercyclic with the same choice of function $f(\cdot).$ Let us finally note that it is not clear how we can apply Proposition \ref{backspace} here in order to see that $(T(t))_{t\geq 0}$ is not $g$-frequently hypercyclic if $g(\cdot)$ is of subexponential growth.
\end{itemize}

\subsection{On frequently hypercyclic translation semigroups on complex sectors}

Our main aim here is to state the following slight extension of \cite[Theorem 6]{kuckanje}:

\begin{theorem}\label{nasser}
Assume that $1\leq p<\infty,$ $\Delta=\Delta(\alpha)$ for some $\alpha \in (0,\pi/2]$ or $\Delta={\mathbb C}$
and
$E=
L^{p}_{\rho}(\Delta, {\mathbb K}).$ Then the following statements are equivalent:
\begin{itemize}
\item[(i)] The translation semigroup $(T(t))_{t\in \Delta}$ is chaotic.
\item[(ii)] There exists a ray $R\subseteq \Delta$ starting at zero such that, for every $m\in {\mathbb N},$ we have
$$
\int_{F_{R,m}}\rho(s) \, ds <\infty, \mbox{ where }F_{R,m}:=\Bigl\{t\in \Delta : d(t,R)=\inf_{s\in R}|t-s|\leq m \Bigr\}.
$$
\item[(iii)] There exists a ray $R\subseteq \Delta$ starting at zero such that the translation semigroup $(T(t))_{t\in R}$ is frequently hypercyclic on $E.$
\end{itemize}
Any of these statements implies that $(T(t))_{t\in \Delta}$ is frequently hypercyclic on $E.$
\end{theorem}

\begin{proof}
The equivalence of (i) and (ii) has been proved in \cite[Theorem 6]{kuckanje}. Assume that (ii) holds with the ray $R:=\{te^{i\varphi} : t\geq 0 \}\subseteq \Delta$ and some angle $\varphi \in [0,2\pi).$ The implication (ii) $\Rightarrow $ (iii) follows by applying the Frequent Hypercyclic Criterion for $C_{0}$-semigroups with the mapping $S(t) : E_{0} \rightarrow E$ defined by $(S(t)f)(x):=f(x-te^{i\varphi}),$ $x\in te^{i\varphi}+supp(f)$ and $(S(t)f)(x):=0,$ otherwise ($f\in E_{0},$ $t>0$); see the proof of \cite[Proposition 3.3]{man-peris}, with appealing to \cite[Lemma 4.2]{kuckanje-prim} in place of \cite[Lemma 4.2]{fund}. Here it is only worth noting that, in the case $p=1,$ the Bochner integrability of mapping $t\mapsto S(te^{i\varphi})f,$ $t\geq 0$ ($f\in E_{0}$) follows essentially from the corresponding part of the proof of \cite[Proposition 3.3]{man-peris} and the fact that $\int^{\infty}_{0}\rho(te^{i\varphi})\, dt$ is convergent, which follows from the next calculation involving \cite[Lemma 4.2]{fund} and the last estimate from the proof of \cite[Theorem 6]{kuckanje}:
\begin{align*}
\int^{\infty}_{0}& \rho\bigl(te^{i\varphi} \bigr)\, dt =\sum_{k=0}^{\infty}\int^{k+1}_{k}\rho\bigl(te^{i\varphi} \bigr)\, dt
\\ & \leq \mbox{Const.} \sum_{k=0}^{\infty} \rho \bigl(kte^{i\varphi}\bigr)\leq \mbox{Const.} \int_{F_{R,m}}\rho(s) \, ds;
\end{align*}
the Pettis integrability of mapping $t\mapsto S(te^{i\varphi})f,$ $t\geq 0$ ($f\in E_{0}$) for $p>1$ follows much easier, by taking functionals. Arguing as in the proof of \cite[Theorem 3.8]{man-marina}, it can be simply shown that the validity of (iii) in the case $\Delta={\mathbb C}$ implies that the backward shift operator $B$ is frequently hypercyclic on the Banach space $l_{p}^{v}:=\{(x_{k})_{k\in {\mathbb Z}} \, ; \, \|(x_{k})_{k\in {\mathbb Z}}\|:=(\sum_{k\in {\mathbb Z}}|x_{k}|^{p}v_{k})^{1/p} < \infty\},$ where $v_{k}:=\rho(ke^{i\varphi})$ for $k\in {\mathbb Z}$ (just use \cite[Lemma 4.2]{kuckanje-prim} and replace the segment $[k,k+1]$ in the proof with the region $ke^{i\varphi}+\Delta_{1}$). By \cite[Theorem 12.3]{bayart-ruse}, we get that $\sum_{k=0}^{\infty}\rho(ke^{i\varphi}) <\infty$ and (i) follows by applying \cite[Theorem 4]{kuckanje}. The proof of implication (iii) $\Rightarrow$ (i), in the case that $\Delta=\Delta(\alpha)$ for some $\alpha \in (0,\pi/2],$ is similar and therefore omitted.
\end{proof}

A simple modification of the proof of \cite[Proposition 3.4]{man-marina} yields the following:

\begin{proposition}\label{qwea}
Assume that $E:=
C_{0,\rho}(\Delta, {\mathbb K})$ and $\lim_{t\in \Delta',|t|\rightarrow \infty}\rho(t)=0.$ Then $(T(t))_{t\in \Delta'}$ is frequently hypercyclic and any single operator $T(t_{0})$ is frequently hypercyclic ($t_{0}\in \Delta'\setminus \{0\}$).
\end{proposition}

As pointed out in \cite[Remark 3.5]{man-marina}, the condition $\lim_{t\in \Delta',|t|\rightarrow \infty}\rho(t)=0$ is only sufficient but not necessary for $(T(t))_{t\in \Delta'}$ to be frequently hypercyclic ($\Delta'=[0,\infty),$ $\Delta={\mathbb R}$). In the case that $\Delta'=\Delta=\Delta(\pi/4),$ a simple counterexample can be obtained by considering the translation semigroup from \cite[Example 4.15]{kuckanje-prim}, with the pivot space being $
C_{0,\rho}(\Delta, {\mathbb K})$.

\subsection{On  semigroups and semiflows}

In this subsection, we turn our attention to $C_{0}$-semigroups induced by semiflows.  For $L^{p}_{\rho_{1}}$ setting, we define $\tilde{\rho_{1}} : {\mathbb C} \rightarrow [0,\infty)$ by $\tilde{\rho_{1}}(x):=\rho_{1}(x),$ $x\in \Delta$ and $\tilde{\rho_{1}}(x):=0;$ otherwise, for $C_{0,\rho}$ setting, the notion of $\tilde{\rho}$ is understood as before.
The following result is very similar to Theorem \ref{dov} and Theorem \ref{dovs}. For the sake of completeness, we will include the most relevant details of proof:

\begin{theorem}\label{athens}
Let $\varphi : \Delta \times \Omega \rightarrow \Omega$ be a
semiflow, $f : [0,\infty) \rightarrow [1,\infty)$ be an increasing mapping and $m_{k}:=f(k),$ $k\in {\mathbb N}_{0}.$
\begin{itemize}
\item[(i)]
Suppose that $1\leq p<\infty,$ $\varphi(t,\cdot)$ is a locally Lipschitz continuous
function for all $t\in \Delta $ and the condition \emph{(i)-(b)} of Lemma \ref{mars} holds. Suppose, further, that $E:=
L^{p}_{\rho_{1}}(\Omega, {\mathbb K}),$ as well as that
for every integer $k\in {\mathbb N}$ and for every compact $K\subseteq \Omega,$ there exist only finite number $c_{k,K}$ of tuples $(n_{1},n_{2})\in {\mathbb N}_{0}^{2}$
such that $n_{1}\neq n_{2},$
\begin{align}\label{mangup}
\varphi\bigl( t_{0} \lfloor m_{k+n_{1}} \rfloor ,K) \cap  \varphi\bigl( t_{0} \lfloor m_{k+n_{2}} \rfloor ,K) \neq \emptyset
\end{align}
\begin{align}\label{mangupe}
\varphi\bigl( t_{0} \lfloor m_{k-n_{1}} \rfloor ,K) \cap  \varphi\bigl( t_{0} \lfloor m_{k-n_{2}} \rfloor ,K) \neq \emptyset , \mbox{ provided }k\geq \max \bigl(n_{1},n_{2}\bigr),
\end{align}
as well as that the sequence $(c_{k,K})_{k\in {\mathbb N}}$ is bounded for every fixed compact $K\subseteq \Omega.$
Let there exist a number $t_{0}\in \Delta \setminus \{0\}$ such that
the following conditions hold for any compact subset $K$ of $\Omega$:
\begin{itemize}
\item[(a)] The sequence $\sum_{n=1}^{k}\int_{K}|\det D\varphi (t_{0}(\lfloor m_{k} \rfloor -\lfloor m_{k-n} \rfloor),x)|\tilde{\rho_{1}}(\varphi(t_{0}(\lfloor m_{k} \rfloor -\lfloor m_{k-n} \rfloor),x)^{-1})\, dx, k\in\mathbb N$ converges, uniformly in $k\in {\mathbb N}.$
\item[(b)] The series $\sum_{n=1}^{\infty}\int_{K}|\det D\varphi (t_{0}(\lfloor m_{k+n} \rfloor -\lfloor m_{k} \rfloor),x)| \tilde{\rho_{1}}(\varphi(t_{0}(\lfloor m_{k+n} \rfloor -\lfloor m_{k} \rfloor),x)^{-1})\, dx$ converges unconditionally, uniformly in $k\in {\mathbb N}.$
\item[(c)] The series $\sum_{n=1}^{\infty}\int_{K}|\det D\varphi (t_{0}\lfloor m_{n} \rfloor,x)|^{-1}\rho_{1}(\varphi(t_{0}\lfloor m_{n} \rfloor,x))\, dx$ converges unconditionally.
\end{itemize}
Then $(T_{\varphi}(t))_{t\in \Delta}$ is $f$-frequently hypercyclic, $(T(te^{i\arg(t_{0})}))_{t\geq 0}$ is $f$-frequently hypercyclic and the operator $T(t_{0})$ is l-$(m_{k})$-frequently hypercyclic.
\item[(ii)] Suppose that the condition \emph{(ii)-(b)} of Lemma \ref{mars} holds and $E:=
C_{0,\rho}(\Omega, {\mathbb K})$. Suppose, further, that for every integer $k\in {\mathbb N}$ and for every compact $K\subseteq \Omega,$ there exist only finite number $c_{k,K}$ of tuples $(n_{1},n_{2})\in {\mathbb N}_{0}^{2}$
such that $n_{1}\neq n_{2}$ and \eqref{mangup}-\eqref{mangupe} hold,
as well as that the sequence $(c_{k,K})_{k\in {\mathbb N}}$ is bounded for every fixed compact $K\subseteq \Omega.$
Let a number $t_{0} \in \Delta \setminus \{0\}$ be such that
the following conditions hold for any compact subset $K$ of $\Omega$:
\begin{itemize}
\item[(a)]  For any $\epsilon>0$ there is an integer
$N\in {\mathbb N}$ such that for any finite set $F\subseteq [N,\infty) \cap {\mathbb N}$ and for any $k\in {\mathbb N}$ we have
$$
\sup_{n\in F,k\geq n, x\in K}\tilde{\rho} \bigl(\varphi(t_{0}(\lfloor m_{k} \rfloor -\lfloor m_{k-n} \rfloor),x)^{-1}\bigr)\in L(0,\epsilon).
$$
\item[(b)]  For any $\epsilon>0$ there is an integer
$N\in {\mathbb N}$ such that for any finite set $F\subseteq [N,\infty) \cap {\mathbb N}$ and for any $k\in {\mathbb N}$ we have
$$
\sup_{1\leq n\leq k,n\in F,x\in K}\tilde{\rho} \bigl(\varphi(t_{0}(\lfloor m_{k+n} \rfloor -\lfloor m_{k} \rfloor),x)^{-1}\bigr) \in L(0,\epsilon).
$$
\item[(c)]  For any $\epsilon>0$ there is an integer
$N\in {\mathbb N}$ such that for any finite set $F\subseteq [N,\infty) \cap {\mathbb N}$ we have
$$
\sup_{n\in F,x\in K}\rho \bigl(\varphi(t_{0}\lfloor m_{n} \rfloor,x)\bigr) \in L(0,\epsilon).
$$
\end{itemize}
Then $(T_{\varphi}(t))_{t\in \Delta}$ is $f$-frequently hypercyclic, $(T(te^{i\arg(t_{0})}))_{t\geq 0}$ is $f$-frequently hypercyclic and the operator $T(t_{0})$ is l-$(m_{k})$-frequently hypercyclic.
\end{itemize}
\end{theorem}

\begin{proof}
The proof of theorem can be deduced by repeating  almost literally the arguments given in the proof of Theorem \ref{dov}, with appropriate technical modifications. Set $E_{0}:=C_{c}(\Omega, {\mathbb K})$ and $S_{\lfloor m_{n} \rfloor} f: \Omega \rightarrow {\mathbb K}$ by $(S_{\lfloor m_{n} \rfloor} f)(x):=f(\varphi(t_{0}\lfloor m_{n} \rfloor,x)^{-1}),$ $x\in \varphi(t_{0}\lfloor m_{n} \rfloor, supp(f))$ and $(S_{\lfloor m_{n} \rfloor} f)(x):=0,$ otherwise ($f\in E_{0},$ $n\in {\mathbb N}$).
Then for each $f\in E_{0}$ one has
$T_{\varphi}(t_{0}\lfloor m_{n} \rfloor)S_{\lfloor m_{n} \rfloor}f=f ,$ $n\in {\mathbb N}$ and
\begin{align*}
T_{\varphi}&(t_{0}\lfloor m_{k} \rfloor)S_{\lfloor m_{n+k} \rfloor}f
\\ & =f\Biggl( \varphi \Bigl(t_{0}\lfloor  m_{k+n} \rfloor ,\varphi(t_{0}\lfloor m_{k}\rfloor ,x)\Bigr)^{-1}\Biggr)\chi_{\varphi \bigl(t_{0}\lfloor m_{k} \rfloor , \varphi(t_{0}\lfloor m_{k+n} \rfloor ,supp(f))\bigr)^{-1}},\quad n,\ k\in {\mathbb N}.
\end{align*}
Due to the conditions (a)-(c) and chain rule, the requirements of Theorem \ref{oma} are satisfied and the part (i) for the space $
L^{p}_{\rho_{1}}(\Omega, {\mathbb K})$ immediately follows. The part (ii) is much easier and can be deduced along the same lines.
\end{proof}

We can also formulate an analogue of Theorem \ref{athens} for the sequence of operators $(t_{n})_{n\in {\mathbb N}}\subseteq \Delta' \setminus\{0\},$ as it has been done in Theorem \ref{dov-niz}-Theorem \ref{dovs-niz}:

\begin{theorem}\label{athens-sequences}
Let $\varphi : \Delta \times \Omega \rightarrow \Omega$ be a
semiflow, $f : [0,\infty) \rightarrow [1,\infty)$ be an increasing mapping and $m_{k}:=f(k),$ $k\in {\mathbb N}_{0}.$ Assume that $(t_{n})_{n\in {\mathbb N}}$ is a sequence in $\Delta' \setminus \{0\},$
$\delta'>0,$ \eqref{krke-rke} holds with the number $\delta$ replaced therein with the number $\delta'$ and
an increasing mapping $g : [1,\infty) \rightarrow [1,\infty)$ satisfies that $|t_{n}|\leq g(n)$ for all $n\in {\mathbb N}.$
\begin{itemize}
\item[(i)]
Suppose that $1\leq p<\infty,$ $\varphi(t,\cdot)$ is a locally Lipschitz continuous
function for all $t\in \Delta $ and the condition \emph{(i)-(b)} of Lemma \ref{mars} holds. Suppose, further, that $E:=
L^{p}_{\rho_{1}}(\Omega, {\mathbb K}),$
for every integer $k\in {\mathbb N}$ and for every compact $K\subseteq \Omega,$ there exist only finite number $c_{k,K}$ of tuples $(n_{1},n_{2})\in {\mathbb N}_{0}^{2}$
such that $n_{1}\neq n_{2},$
\begin{align}\label{mangup-si}
\varphi\bigl(   t_{\lfloor m_{k+n_{1}}\rfloor }  ,K) \cap  \varphi\bigl(  t_{ \lfloor m_{k+n_{2}}\rfloor}  ,K) \neq \emptyset
\end{align}
\begin{align}\label{mangupe-si}
\varphi\bigl(  t_{\lfloor m_{k-n_{1}} \rfloor}  ,K) \cap  \varphi\bigl(t_{ \lfloor  m_{k-n_{2}}\rfloor}  ,K) \neq \emptyset , \mbox{ provided }k\geq \max \bigl(n_{1},n_{2}\bigr),
\end{align}
as well as that the sequence $(c_{k,K})_{k\in {\mathbb N}}$ is bounded for every fixed compact $K\subseteq \Omega.$
Let
the following conditions hold for any compact subset $K$ of $\Omega$:
\begin{itemize}
\item[(a)] The series $\sum_{n=1}^{k}\int_{K}|\det D\varphi ( t_{\lfloor m_{k}  \rfloor} -t_{\lfloor m_{k-n} \rfloor},x)| \tilde{\rho_{1}}(\varphi(t_{\lfloor m_{k}  \rfloor} -t_{\lfloor m_{k-n} \rfloor},x)^{-1})\, dx, k\in\mathbb N$ converges unconditionally, uniformly in $k\in {\mathbb N}.$
\item[(b)] The series $\sum_{n=1}^{\infty}\int_{K}|\det D\varphi ( t_{\lfloor m_{k+n}  \rfloor} -t_{\lfloor m_{k} \rfloor},x)| \tilde{\rho_{1}}(\varphi(t_{\lfloor m_{k+n}  \rfloor} -t_{\lfloor m_{k} \rfloor},x)^{-1})\, dx$ converges unconditionally, uniformly in $k\in {\mathbb N}.$
\item[(c)] The series $\sum_{n=1}^{\infty}\int_{K}|\det D\varphi (t_{\lfloor m_{n} \rfloor},x)|^{-1}\rho_{1}(\varphi(t_{\lfloor m_{n} \rfloor},x))\, dx$ converges unconditionally.
\end{itemize}
Then $(T_{\varphi}(t))_{t\in \Delta'}$ is $(g\circ f)$-frequently hypercyclic and the sequence of operators
$(T_{\varphi, n}:=T_{\varphi}(t_{n}))_{n\in {\mathbb N}}$ is l-$(m_{k})$-frequently hypercyclic.
\item[(ii)] Suppose that the condition \emph{(ii)-(b)} of Lemma \ref{mars} holds and $E:=
C_{0,\rho}(\Omega, {\mathbb K})$. Suppose, further, that for every integer $k\in {\mathbb N}$ and for every compact $K\subseteq \Omega,$ there exist only finite number $c_{k,K}$ of tuples $(n_{1},n_{2})\in {\mathbb N}_{0}^{2}$
such that $n_{1}\neq n_{2}$ and \eqref{mangup-si}-\eqref{mangupe-si} hold,
as well as that the sequence $(c_{k,K})_{k\in {\mathbb N}}$ is bounded for every fixed compact $K\subseteq \Omega.$
Let
the following conditions hold for any compact subset $K$ of $\Omega$:
\begin{itemize}
\item[(a)]  For any $\epsilon>0$ there is an integer
$N\in {\mathbb N}$ such that for any finite set $F\subseteq [N,\infty) \cap {\mathbb N}$ and for any $k\in {\mathbb N}$ we have
$$
\sup_{n\in F,k\geq n, x\in K}\tilde{\rho} \bigl(\varphi(t_{\lfloor m_{k} \rfloor} -t_{\lfloor m_{k-n} \rfloor},x)^{-1}\bigr)\in L(0,\epsilon).
$$
\item[(b)]  For any $\epsilon>0$ there is an integer
$N\in {\mathbb N}$ such that for any finite set $F\subseteq [N,\infty) \cap {\mathbb N}$ and for any $k\in {\mathbb N}$ we have
$$
\sup_{1\leq n\leq k,n\in F,x\in K}\tilde{\rho} \bigl(\varphi(t_{\lfloor m_{k+n} \rfloor} -t_{\lfloor m_{k} \rfloor},x)^{-1}\bigr) \in L(0,\epsilon).
$$
\item[(c)]  For any $\epsilon>0$ there is an integer
$N\in {\mathbb N}$ such that for any finite set $F\subseteq [N,\infty) \cap {\mathbb N}$ we have
$$
\sup_{n\in F,x\in K}\rho \bigl(\varphi(t_{\lfloor m_{n} \rfloor},x)\bigr) \in L(0,\epsilon).
$$
\end{itemize}
Then $(T_{\varphi}(t))_{t\in \Delta'}$ is $(g\circ f)$-frequently hypercyclic and the sequence of operators
$(T_{\varphi, n}:=T_{\varphi}(t_{n}))_{n\in {\mathbb N}}$ is l-$(m_{k})$-frequently hypercyclic.
\end{itemize}
\end{theorem}

\subsection{Examples III}
\label{svinjo}
\begin{itemize}
\item[(i)] (see \cite[Example 3.20]{kalmes}) Let $p\geq 1,$ $\Omega:=(0,1),$ $\Delta=\Delta':=[0,\infty),$ $E:=L^{p}_{\rho_{1}}(\Omega,{\mathbb K})$ and
$$
\varphi(t,x):=\frac{x}{x+(1-x)e^{-t}},\quad t\geq 0,\ x\in (0,1).
$$
Then $\varphi(t,\cdot)$ is a continuously differentiable
function for all $t\geq 0 ,$ $D\varphi(t,x)=e^{-t}( x+(1-x)e^{-t})^{-2}$ for all $t\geq 0$ and $x\in
(0,1),$ $\varphi(t,x)^{-1}=xe^{-t}(xe^{-t}+1-x)^{-1}$ for all $t\geq 0$ and $x\in
(0,1),$ $D\varphi(t,x)^{-1}=e^{-t}( xe^{-t}+1-x)^{-2}$ for all $t\geq 0$ and $x\in
(0,1),$
and the condition (i)-(b) of Lemma \ref{mars}(i) holds iff there exist $M\geq 1$ and $\omega \geq 0$ such that
$$
\rho_{1}(x)e^{t}\bigl( x+(1-x)e^{-t} \bigr)^{2}\leq Me^{\omega t}\rho_{1}\Bigl( x\bigl( x+(1-x)e^{-t}\bigr)^{-1}\Bigr),\quad t>0,\  a.e.\ x\in (0,1).
$$
It can be easily seen that the last estimate holds for $\rho_{1}(x):=1/x,$ $x\in (0,1),$ with $M=1$ and $\omega=3,$ so that $(T_{\varphi}(t))_{t\geq 0}$ is a $C_{0}$-semigroup on $E.$ Using \cite[Theorem 2, Remark 3]{kalmes-prim}, it readily follows that  $(T_{\varphi}(t))_{t\geq 0}$ is neither frequently hypercyclic nor chaotic. With the help of Theorem \ref{athens}, we will prove that $(T_{\varphi}(t))_{t\geq 0}$ is $q$-frequently hypercyclic for any number $q>1.$ Let $t_{0}=1,$
let an integer $k\in {\mathbb N}$ and a compact $K=[a,b]\subseteq (0,1)$ be given. Assume that
$x\in \varphi\bigl(  \lfloor (k+n_{1})^{q} +1 \rfloor ,K) \cap  \varphi\bigl(  \lfloor (k+n_{2})^{q} +1\rfloor ,K).$
This implies the existence of numbers $x_{1},\ x_{2}\in [a,b]$ such that
$$
\frac{x_{1}}{x_{1}+(1-x_{1})e^{-(\lfloor (k+n_{1})^{q} +1 \rfloor+1)}}=\frac{x_{2}}{x_{2}+(1-x_{2})e^{-(\lfloor (k+n_{2})^{q} +1 \rfloor+1)}}.
$$
This yields
$$
\frac{x_{1}(1-x_{2})}{x_{2}(1-x_{1})}=e^{\lfloor (k+n_{1})^{q} +1 \rfloor -\lfloor (k+n_{2})^{q} +1 \rfloor},
$$
so that there exists a finite constant $\delta>0$ such that $ (k+n_{1})^{q}  -(k+n_{2})^{q} \leq \delta.$ Arguing as in Remark \ref{hnme}, we get that there exist only finite number $c_{k,K}'$ of tuples $(n_{1},n_{2})\in {\mathbb N}_{0}^{2}$
such that $n_{1}\neq n_{2}$ and \eqref{mangup} hold as well as that the sequence $(c_{k,K}')_{k\in {\mathbb N}}$ is bounded. We can similarly prove that there exist only finite number $c_{k,K}''$ of tuples $(n_{1},n_{2})\in {\mathbb N}_{0}^{2}$
such that $n_{1}\neq n_{2}$ and \eqref{mangupe} hold as well as that the sequence $(c_{k,K}'')_{k\in {\mathbb N}}$ is bounded.
Hence, the requirements necessary for applying Theorem \ref{athens} are fulfilled. The conditions (a)-(c) also hold and we will verify this only for the condition (b), i.e., we will prove that
the series $\sum_{n=1}^{\infty}\int_{a}^{b}|D\varphi ((\lfloor (k+n)^{q}+1 \rfloor -\lfloor k^{q}+1 \rfloor),x)| \tilde{\rho_{1}}(\varphi((\lfloor (k+n)^{q} \rfloor -\lfloor k^{q} \rfloor),x)^{-1})\, dx$ converges unconditionally, uniformly in $k\in {\mathbb N}.$ This follows from the following calculation:
\begin{align*}
\sum_{n=1}^{\infty}&\int_{a}^{b}\Bigl|D\varphi \bigl(\lfloor (k+n)^{q} \rfloor -\lfloor k^{q} \rfloor ,x\bigr)\Bigr| \tilde{\rho_{1}}\bigl(\varphi(\lfloor (k+n)^{q} \rfloor -\lfloor k^{q} \rfloor,x)^{-1}\bigr)\, dx
\\ & =\sum_{n=1}^{\infty}\int_{a}^{b}\frac{xe^{-2(\lfloor (k+n)^{q} \rfloor -\lfloor k^{q} \rfloor)}}{(xe^{-(\lfloor (k+n)^{q} \rfloor -\lfloor k^{q} \rfloor)}+1-x)(x+(1-x)e^{-(\lfloor (k+n)^{q} \rfloor -\lfloor k^{q} \rfloor)})^{2}}\, dx
\\ &=\sum_{n=1}^{\infty}\int_{a}^{b}\frac{x\, dx}{(xe^{-(\lfloor (k+n)^{q} \rfloor -\lfloor k^{q} \rfloor)}+1-x)(e^{\lfloor (k+n)^{q} \rfloor -\lfloor k^{q} \rfloor}x+1-x)}
\\ & \leq \frac{b(b-a)}{(1-b)a}\sum_{n=1}^{\infty}e^{-(\lfloor (k+n)^{q} \rfloor -\lfloor k^{q} \rfloor)}.
\end{align*}
\item[(ii)] Let $p\in [1,\infty),$ $\alpha \in (0,\pi/2],$ let $\Delta=[0,\infty)$ or $\Delta=\Delta(\alpha)$ and let $\Omega:=\Delta^{\circ}$ (the interior of $\Delta $). Suppose that $F : \Omega \rightarrow \Omega$ is a continuously differentiable bijective mapping together with its inverse mapping $F^{-1} : \Omega \rightarrow \Omega.$ Define $\varphi(t,x):=F^{-1}(t+F(x)),$ $t\in \Delta,$ $x\in \Omega.$ Then it can be simply shown that $\varphi(\cdot,\cdot)$ is a semiflow. Let $E:=L^{p}_{\rho_{1}}(\Delta,{\mathbb K}),$ $t_{0}\in \Delta \setminus \{0\},$ $q>1$ and $f(s):=s^{q}+1,$ $s\geq 0.$ The condition \eqref{sufficient} is fulfilled iff
\begin{equation}\label{sufficient-primonja}
\exists M, \omega \in {\mathbb R}\  \forall t \in \Delta
: \rho_{1}(x) \leq  Me^{\omega
|t|}\rho_{1}\bigl(F^{-1}(t+F(x))\bigr)\Bigr|\det DF^{-1}(t+F(x))\Bigr|
\end{equation}
for a.e. $x\in \Omega,$
when $(T_{\varphi}(t))_{t\in \Delta}$ is a $C_{0}$-semigroup on $E.$ It is clear that
\begin{align}\label{kvo}
\varphi(t,x)^{-1}=y\mbox{ iff }y=F^{-1}(F(x)-t)\ \ \mbox{ for }\ \ t\in \Delta,\quad x,\ y\in \Omega.
\end{align}
Arguing similarly as in Remark \ref{hnme} and Example II(i), it can be simply shown that the conditions \eqref{mangup}-\eqref{mangupe} hold true, as well as that the series in the formulations of conditions (a)-(b) of Theorem \ref{athens} converge unconditionally, uniformly in $k\in {\mathbb N}$ since, for a given compact set $K\subseteq \Omega,$ the sums in their definitions are finite and equal zero
for any sufficiently large number $k\geq k_{0}(K).$ Therefore, $(T_{\varphi}(t))_{t\in \Delta}$ will be $f$-frequently hypercyclic, $(T(te^{i\arg(t_{0})}))_{t\geq 0}$ will be $f$-frequently hypercyclic and the operator $T(t_{0})$ will be l-$(m_{k})$-frequently hypercyclic if
the series in the formulation of condition (c) of Theorem \ref{athens} converges unconditionally, i.e., if the series
\begin{align}\label{lade-rade}
\sum_{n=1}^{\infty}\int_{K}\Bigl|\det DF^{-1} \bigl(t_{0}\lfloor n^{q} \rfloor +F(x)\bigr )\Bigr|^{-1}\rho_{1}\Bigl(F^{-1}\bigl(t_{0}\lfloor n^{q} \rfloor +F(x)\bigr)\Bigr)\, dx
\end{align}
converges unconditionally. Let us examine some concrete cases in which the conditions \eqref{sufficient-primonja} and \eqref{lade-rade} hold true with the sector $\Delta=\Delta(\alpha)$:
\begin{itemize}
\item[(a)] Let $c>0$ and $F(x+iy):=c(x\pm iy),$ $x+iy\in \Omega.$ If $\rho_{1}(\cdot)$ is an admissible weight function on $\Delta$ and $\sum_{n=1}^{\infty}\rho_{1}(t_{0}n^{q}/c)<\infty $ for the choice of sign $+,$ resp., $\sum_{n=1}^{\infty}\rho_{1}(\overline{t_{0}}n^{q}/c)<\infty $ for the choice of sign $-.$ Then it can be easily seen that \eqref{sufficient-primonja}-\eqref{lade-rade} are valid. An application of Theorem \ref{athens-sequences} can be also made provided that the conditions $\sum_{n=1}^{\infty}\rho_{1}(t_{n}/c)<\infty $ or $\sum_{n=1}^{\infty}\rho_{1}(\overline{t_{n}}/c)<\infty $
hold for an appropriate sequence $(t_{n})_{n\in {\mathbb N}}$ in $\Delta'\setminus \{0\}.$
\item[(b)] Let $F(x+iy):=(x+iy)^{a}e^{i\alpha(1-a)},$ $x+iy\in \Delta$ for some number $a\in (0,1).$
Then $\varphi(t,x+iy)=e^{i\alpha\frac{a-1}{a}}(t+(x+iy)^{a}e^{i\alpha(1-a)})^{1/a}$ for $t,\ x+iy\in \Delta$ and
\begin{align*}
\Bigl|\det & DF^{-1} \bigl(t +F(x)\bigr )\Bigr|
\\ & \sim \mbox{Const.} \Bigl| t+(x+iy)^{a}\Bigr|^{\frac{2}{a}-2}\bigl| x+iy\bigr|^{2a-2},\quad t\in \Delta, \ x+iy\in \Delta.
\end{align*}
Using this estimate, it can be simply verified that the conditions \eqref{sufficient-primonja} and \eqref{lade-rade} hold provided that $\zeta>0,$ $\rho_{1}(x+iy):=(1+|x+iy|)^{-\zeta},$ $x+iy\in \Delta$ and $q(2-\frac{2}{a})-\frac{q}{a}\zeta<-1.$
\end{itemize}
\item[(iii)] (\cite{kalmes-prim}) Let $p\in [1,\infty)$ and let $E:=L^{p}[0,1].$ Assume that $h\in C[0,1]$ and $(h(x)-\Re h(0))/x\in L^{1}[0,1].$ It is well known that the $C_{0}$-semigroup $(T(t))_{t\geq 0}$ governing the solutions of (linear) von Foerster-Lasota equation
\begin{align}\label{kalme}
u_{t}(t,x)=-xu_{x}(t,x)+h(x)u(t,x),\quad t\geq 0,\ x\in (0,1); \ \ u(0,x)=u_{0}(x),
\end{align}
is hypercyclic iff it is chaotic iff it is frequently hypercyclic iff $\Re (h(0))\geq (-1)/p.$ Let $f : [0,\infty) \rightarrow [1,\infty)$ be an increasing mapping such that $\liminf_{x\rightarrow \infty}f(x)/x>0.$ Then freqent hypercyclicity implies $f$-frequent hypercyclicity which further implies hypercyclicity, so that any of above conditions is necessary and sufficient for $f$-frequent hypercyclicity of  solutions to \eqref{kalme}. This example particularly shows that finding the necessary and sufficient conditions for $f$-frequent hypercyclicity of $C_{0}$-semigroups induced by semiflows is far from being a trivial problem.
\end{itemize}

In contrast to Theorem \ref{athens}, Theorem \ref{oma} and \cite[Theorem 2.2]{man-peris} can be always applied for proving frequent hypercyclicity of $C_{0}$-semigroups induced by semiflows:

\subsection{Examples IV}
\label{ghf}
\begin{itemize}
\item[(i)] (\cite[Example 3.1.28(v)]{knjigaho})
Assume that $\Delta :=[0,\infty),$ $\Omega :=\{(x,y) \in {{\mathbb R}^{2}} : x^{2}+y^{2}>1 \},$
$p>0,$ $q\in {\mathbb R}$ and 
$$
\varphi(t,x,y):=e^{pt}(x\cos qt -y\sin qt, x\sin
qt +y\cos qt),\ t\geq 0,\ (x,y)\in \Omega .
$$
Then
$\varphi : \Delta \times
\Omega \rightarrow \Omega$ is a semiflow and we already know that
$(T_{\varphi}(t))_{t\geq 0}$ is topologically mixing and chaotic in the Fr\'echet space $C(\Omega, {\mathbb K}).$ Using Theorem \ref{oma}, it can be simply shown that $(T(t))_{t\geq 0}$ is frequently hypercyclic, as well.
\item[(ii)] (\cite[Example 5.5]{kalmes}) Let $\Delta:=[0,\infty),$ $E:=L^{p}_{\rho_{1}}({\mathbb R}^{n}),$ $\rho_{1}(x):=1/(1+|x|^{2})$ ($x\in {\mathbb R}^{n}$) and $\varphi(t,x):=e^{t}x$ ($t\geq 0,$ $x\in {\mathbb R}^{n}$). We already know that the $C_{0}$-semigroup $(T_{\varphi}(t))_{t\geq 0}$ is chaotic. Using \cite[Theorem 2.2]{man-peris}, it can be simply proved that $(T_{\varphi}(t))_{t\geq 0}$ is frequently hypercyclic, as well. The strongly continuous (weighted composition) semigroups considered in \cite[Example 3.20, Example 4.13]{kalmes} are also frequently hypercyclic, which can be shown by applying \cite[Theorem 2]{kalmes-prim}.
\item[(iii)] (\cite[Example 3.1.41(i)]{knjigaho}) Assume that $p\in [1,\infty),$ $\alpha \in (0,\frac{\pi}{2}],$
$\Delta\in \{[0,\infty),\ \Delta(\alpha)\},$ $\Omega=(1,\infty)$ and $0<\alpha_{1}\leq 1.$ Define $\varphi : \Delta \times \Omega
\rightarrow \Omega$ and $\rho_{1} :
\Omega \rightarrow (0,\infty)$ through:
\begin{equation}\label{ekv}
\varphi(t,x):=(\Re t+x^{\alpha_{1}})^{1/\alpha_{1}}
\mbox{ and }\rho_{1}(x):=e^{-x^{\alpha_{1}}},\ t\in \Delta,\ x\in
\Omega.
\end{equation}
It is straightforward to verify that $\varphi(\cdot,\cdot)$ is a
semiflow and
$(T_{\varphi}(t))_{t\in \Delta}$ is a $C_{0}$-semigroup on
$L^{p}_{\rho_{1}}(\Omega,{\mathbb K}).$ Employing \cite[Theorem 2.2]{man-peris}, we can simply prove that $(T_{\varphi}(t))_{t\in R}$ is frequently hypercyclic for any ray $R\subseteq \Delta,$ so that $(T_{\varphi}(t))_{t\in \Delta}$ is frequently hypercyclic, as well.
\end{itemize}

\subsection{Open problems}

We propose two open problems motivated by reading the paper \cite{erdos} by K.-G. Grosse-Erdmann.
Let us recall that, for every operator $T\in L(E)$ such that $T^{-1}\in L(E),$ the hypercyclicity of $T$ implies the hypercyclicity of $T^{-1};$ see e.g. \cite{Grosse}. As pointed out in \cite[Problem 9]{erdos}, the corresponding question for frequent hypercyclicity is still unsolved. On the other hand, in \cite[Theorem 2.5]{fund}, W. Desch, W. Schappacher and G. F. Webb have proved that for any $C_{0}$-group of linear operators $(T(t))_{t\in {\mathbb R}}$ on a separable Banach space $E,$ the following assertions are equivalent:
\begin{itemize}
\item[(i)] the semigroup $(T(t))_{t\geq 0}$ is hypercyclic;
\item[(ii)] the semigroup $(T(-t))_{t\geq 0}$ is hypercyclic;
\item[(iii)] there exists some $x \in E$ such that both sets $\{T(t)x : t\geq 0\}$ and $\{T(-t)x : t\geq 0\}$
are dense in $E.$
\end{itemize}
A similar assertion holds in separable Fr\'echet spaces (see e.g. \cite[Theorem 3.1.2(iv), Theorem 3.1.4(ii)]{knjigaho} and the proof of \cite[Theorem 2.5]{fund}) and, based on the above discussion, we would like to propose the following problems:\vspace{0.1cm}

{\sc Problem 1.} Let $(T(t))_{t\in {\mathbb R}}$ be a $C_{0}$-group of linear operators $(T(t))_{t\in {\mathbb R}}$ on a separable Fr\'echet space $E.$ Is it true that the following assertions are equivalent:
\begin{itemize}
\item[(i)] the semigroup $(T(t))_{t\geq 0}$ is frequently hypercyclic;
\item[(ii)] the semigroup $(T(-t))_{t\geq 0}$ is frequently hypercyclic;
\item[(iii)] there exists some $x \in E$ that is both frequently hypercyclic vector for $(T(t))_{t\geq 0}$ and frequently hypercyclic vector for $(T(-t))_{t\geq 0}?$
\end{itemize}
\vspace{0.1cm}

{\sc Problem 2.} Let $(T(t))_{t\in {\mathbb R}}$ be a $C_{0}$-group of linear operators $(T(t))_{t\in {\mathbb R}}$ on a separable Fr\'echet space $E.$ Profile the class of increasing functions $f : [0,\infty) \rightarrow [1,\infty)$ for which the following assertions are equivalent:
\begin{itemize}
\item[(i)] the semigroup $(T(t))_{t\geq 0}$ is $f$-frequently hypercyclic;
\item[(ii)] the semigroup $(T(-t))_{t\geq 0}$ is $f$-frequently hypercyclic;
\item[(iii)] there exists some $x \in E$ that is both $f$-frequently hypercyclic vector for $(T(t))_{t\geq 0}$ and $f$-frequently hypercyclic vector for $(T(-t))_{t\geq 0}.$
\end{itemize}

The notion of $f$-frequent hypercyclicity is still very unexplored and we can propose a great number of other problems for $C_{0}$-semigroups and single operators in Fr\' echet spaces.


\end{document}